\documentclass[11pt, a4paper]{article}
 \pdfoutput=1
\usepackage{array,color,colortbl} 
\usepackage{graphics}
\usepackage[toc,page]{appendix}
\usepackage{float}
\usepackage{amsmath}
\usepackage{amsthm}
\usepackage{url}
\usepackage{units}
\usepackage{float}%
\usepackage{makebox}%
\usepackage{relsize}%
\usepackage{empheq}%
\usepackage{booktabs} 
\usepackage{array} 
\usepackage{paralist} 
\usepackage{verbatim} 
\usepackage{subfig} 
\usepackage{mathtools}
\usepackage{blkarray}
\usepackage{lscape}
\usepackage[T1]{fontenc}
\usepackage[utf8]{inputenc}
\usepackage{authblk}
\usepackage{arydshln}
\usepackage{lscape}
\usepackage{amssymb}
\usepackage{enumitem}
\usepackage[nobottomtitles*]{titlesec}

\newcolumntype{P}[1]{>{\centering\arraybackslash}p{#1}}

\theoremstyle{definition}
\newtheorem{lemma}{Lemma}

\newtheorem{example}{Example}
\newtheorem{theorem}{Theorem}
\newtheorem{corollary}{Corollary}
\newtheorem{conjecture}{Conjecture}

\title{Row-column factorial designs with strength at least $2$}

\author{Fahim Rahim and Nicholas J. Cavenagh}
\affil{Department of Mathematics and Statistics, \\
University of Waikato, Private Bag 3105, \\
Waikato Mail Centre 3240, New Zealand}

\begin{document}
	
\maketitle

\begin{abstract}
The $q^k$ (full) factorial design with replication $\lambda$ is the multi-set consisting of $\lambda$ occurrences of each element of each $q$-ary vector of length $k$; we denote this by $\lambda\times [q]^k$. 
An $m\times n$ {\em row-column factorial design} $q^k$ of {\em strength} $t$ is an arrangement of the elements of $\lambda \times [q]^k$ into an $m\times n$ array (which we say is of type $I_k(m,n,q,t)$)   
such that for each row (column), the set of vectors therein are the rows of an orthogonal array of degree $k$, size $n$ (respectively, $m$), $q$ levels and strength $t$. 
Such arrays are used in experimental design. In this context, for a row-column factorial design of strength $t$, all subsets of interactions of size at most $t$ can be estimated without confounding by the row and column blocking factors. 

In this manuscript we study row-column factorial designs with strength $t\geq 2$.  
Our results for strength $t=2$ are as follows. 
For any prime power $q$ and assuming $2\leq M\leq N$, we show that 
 there exists an array of type $I_k(q^M,q^N,q,2)$ if and only if $k\leq M+N$, 
$k\leq (q^M-1)/(q-1)$ and $(k,M,q)\neq (3,2,2)$. 
We find necessary and sufficient conditions for the existence of $I_{k}(4m,n,2,2)$ for small parameters. We 
also show that  $I_{k+\alpha}(2^{\alpha}b,2^k,2,2)$ exists whenever
$\alpha\geq 2$ and $2^{\alpha}+\alpha+1\leq k<2^{\alpha}b-\alpha$, assuming there exists a Hadamard matrix of order $4b$. 

For $t=3$ we focus on the binary case. Assuming $M\leq N$, there exists an array of type $I_k(2^M,2^N,2,3)$ if and only if $M\geq 5$,  $k\leq M+N$ and $k\leq 2^{M-1}$.
 Most of our constructions use linear algebra, often  in application to existing orthogonal arrays and Hadamard matrices. 
\end{abstract}

	
	
\section{Introduction}

For any integer $s$,  let $[s]=\{0,1,\dots ,s-1\}$. 
An {\em orthogonal array} of size $N$, degree $k$, $q$ levels and strength $t$, denoted OA$(N,k,q,t)$ is an $N\times k$ array with entries from $[q]$ such that in every $N\times t$ submatrix, every $1\times t$ row vector appears $N/q^t$ times. 
The $q^k$ (full) factorial design with replication $\lambda$ is the multi-set consisting of $\lambda$ occurrences of each element of $[q]^k$; we denote this by $\lambda\times [q]^k$. 

An $m\times n$ {\em row-column factorial design} $q^k$ is any  arrangement of the elements of $\lambda \times [q]^k$ into an $m\times n$ array. 
We say that such an array has {\em strength} $t$ if for each row (column), the set of vectors therein are the rows of an orthogonal array of size $k$, degree $n$ (respectively, $m$), $q$ levels and strength $t$. 
That is, if we consider any subset of $t$ positions within the vectors in a fixed row (or column), we obtain 
 a $[q]^t$ full-factorial design with replication $n/q^t$ (respectively, $m/q^t$).
We denote such a row-column factorial design by $I_k(m,n,q,t)$, where the replication number or index $\lambda $ of the design is given by $ \lambda = mn/q^k$.

For example, in Table \ref{tabble1} the elements of $[3]^4$ are arranged into a $9\times 9$ array such that the vectors in each row and column are the rows of an OA$(9,4,3,2)$. Thus this is an array of type $I_4(9,9,3,2)$. 

 \begin{table}[H] 
		\begin{center}
			\renewcommand{\arraystretch}{1.2}
			\begin{tabular}{ccccccccc}
				0000&1011&2022&0112&1120&2101&0221&1202&2210 \\
				0111&1122&2100&0220&1201&2212&0002&1010&2021 \\
				0222&1200&2211&0001&1012&2020&0110&1121&2102 \\
				1021&2002&0010&1100&2111&0122&1212&2220&0201 \\
				1102&2110&0121&1211&2222&0200&1020&2001&0012 \\
				1210&2221&0202&1022&2000&0011&1101&2112&0120 \\
				2012&0020&1001&2121&0102&1110&2200&0211&1222 \\
				2120&0101&1112&2202&0210&1221&2011&0022&1000 \\
				2201&0212&1220&2010&0021&1002&2122&0100&1111 \\
	
			\end{tabular}
			\caption{A row-column factorial design $I_4(9,9,3,2)$.}
			\label{tabble1}
		\end{center}
	\end{table}

Within experimental design, a {\em row-column} design can refer to a variety of combinatorial designs, all with the property of being arranged in a rectangular array, where regularity conditions may be imposed in order to estimate effects without confounding. 
Table \ref{tabble1}, for example, could be used to study the effects of 4 drugs on cows, each at 3 dosage levels while controlling for the effects of 9 breeds (the rows) and 9 age groups (the columns). Here the vector $(2,0,2,2)$ in the first row and third column indicates that the first breed and third age group are given the highest dosage of the first, third and fourth drug and the lowest dosage of the second drug. The property of being an array of type $I_4(9,9,3,2)$ eliminates, for example, confounding between breed or age group and the interaction between any pair of drugs.  
We refer the reader to \cite{godolphin2019construction, rahim2021row-column} for a literature review on the application of row-column factorial designs to statistical experimental design. 

In this paper two arrays are {\em equivalent} under any: (a) reordering of rows; (b) reordering of columns; (c) reordering of levels (applied globally); and (d) reordering of the entries in each vector (with the same reordering applied globally). It is also convenient to use the terms  ``array'' and ``matrix'' interchangeably. When linear algebra is applied we often work over the field of order $q$, with an understanding that when $q$ is a prime power, the levels are relabelled with $[q]$ as a final step.   

By definition, if there exists an array of type $I_k(m,n,q,t)$ then 
$q^k\vert mn$ and  there exists an OA$(m,k,q,t)$ and there exists an OA$(n,k,q,t)$.
 If $(k,m,n,q,t)$ is a $5$-tuple satisfying these three necessary conditions we say
that $(k,m,n,q,t)$ is {\em admissible}. 

Necessary conditions for the existence of orthogonal arrays of strength $t$
 imply further necessary conditions for the existence of row-column factorial designs of strength $t$. It is impractical to list all known necessary conditions (in particular as $t$ grows large);  we refer the reader to 
surveys in III.6 and III.7 of the Handbook of Combinatorial Designs \cite{Colbourn:2006:HCD:1202540}.  
Elementary conditions imply that $t\leq k$ and $q^t$ divides both $m$ and $n$. 
In summary:
\begin{lemma}
If $(k,m,n,q,t)$ is admissible then $q^t|m$, $q^t|n$, $q^k|mn$ and $t\leq k$.
\label{triviality} 
\end{lemma}

Necessary and sufficient conditions for a row-column factorial design of strength $1$ are given in \cite{rahim2021row-column}, generalizing \cite{godolphin2019construction} and \cite{wang2017orthogonal}.   
 
\begin{theorem} \label{thm:prevresult} (\cite{rahim2021row-column})
	Let $m\leq n$. 
		There exists $I_k(m,n,q,1)$ (that is, an $m\times n$ row-column factorial design $q^k$ of strength $1$) if and only if:  
			\begin{enumerate}
			\item[i.]  $q|m$ and $q|n$; 
			\item[ii.] if $k=q=m=2$ then $4$ divides $n$; and 
			\item[iii.] $(k,m,n,q)\neq (2,6,6,6)$.
		\end{enumerate}
	\end{theorem}

Note that an array $I_k(n,n,q,t)$ implies the existence of a set of $k$ mutually orthogonal frequency squares (MOFS) of size $n$ based on a set of size $q$. Thus, the existence of row-column factorial designs also relates to the existence of frequency squares and Latin squares. For example, the exceptions in the previous theorem include pairs of orthogonal Latin squares of orders $2$ and $6$, which are well-known not to exist. Some results, including a table of lower bounds, related to the existence of MOFS can be found in \cite{cavenaghwanless2020,laywine2001table, li2014some}.


In this manuscript we focus on row-column factorial designs of strength $2$ and higher.
  Binary row-column factorial designs of strength $1$ which come as close as possible to strength $2$ are studied in \cite{godolphin2019construction}, in the case when the dimensions of the array are powers of $2$. The motivation in \cite{godolphin2019construction} is to be able to estimate as many two-factor interactions as possible without confounding, given fixed parameters. 
 
In the binary strength $2$ case we will frequently make use of Hadamard matrices. A  
\textit{Hadamard matrix} $H(n)$ is a square matrix of order $n$, having entries from the set $ \{1,-1\} $ such that any two rows are orthogonal, i.e., it satisfies the equation:
$H(n)H(n)^T=nI_n.$
 If a Hadamard matrix $H(n)$ exists, then either $n=2$ or $n$ is divisible by $4$. However, the converse is an open problem known as the {\em Hadamard conjecture}; the smallest value for which it is not known whether a Hadamard matrix exists or not is 668 \cite{djokovic2007hadamard}.
 
  The following lemma gives a relationship between binary orthogonal arrays of strength 2 and 3 and the Hadamard matrices of order $4m.$
\begin{lemma} \cite[p.~148]{hedayat1999orthogonal}
Let $m\geq 4$. Orthogonal arrays OA$(4m, 4m - 1,2,2)$ and OA$(8m, 4m, 2,3)$ exist if and only if there exists a Hadamard matrix order $4m$.
\label{haddd} 
\end{lemma}

It is worth mentioning how orthogonal arrays can be constructed from Hadamard matrices as per the previous lemma, as this idea is frequently applied in Sections 5 and 6, where we focus on binary arrays.  
Let $H$ be a Hadamard matrix of order $4m$. Assume that $H$ is in {\em normalized form}; that is, we assume the first row and column of $H$ only contain the entry $1$. Now delete the first column and replace each $-1$ with $0$. The resultant array is 
an OA$(4m,4m-1,2,2)$. Next, consider the array $[H|-H]^T$  
and again replace each $-1$ with $0$; the resultant array is an OA$(8m,4m,2,3)$. The rows of such an array are the codewords of a code known as a {\em Hadamard code} \cite{horadam2012hadamard}. 

For a binary vector ${\mathbf v}$, we often say that its {\em weight} $\omega({\mathbf v})$ is equal to the number of $1$'s in ${\mathbf v}$. 
It is also often convenient to say that two binary vectors ${\bf v}$ and ${\bf w}$ of length $4k$ are {\em orthogonal}  if each has weight $2k$ and ${\bf v}\cdot {\bf w}=k$. So in a binary orthogonal 
array of strength $2$, each pair of columns is necessarily orthogonal. 

The previous lemma, together with the Bose-Bush bound for orthogonal arrays (\cite{Colbourn:2006:HCD:1202540}; \cite{plackettburman} originally) implies the following necessary conditions for strength $2$ row-column factorial designs.  

\begin{lemma}
Let $m\leq n$. If there exists an array of type $I_k(m,n,q,2)$, 
 then $k\leq (m-1)/(q-1)$.  If 
 there exists an array of type $I_{m-1}(m,n,2,2)$, then there is a Hadamard matrix of order $m$. 
 \label{2strong}
 \end{lemma}

Since the Hadamard conjecture is a well-studied but unsolved open problem, it is likely that generalizing 
Theorem \ref{thm:prevresult}
 to the strength $2$ case, that is finding necessary and sufficient conditions for the existence of a row-column factorial design of strength $2$, is untenable even in the binary case. 
 
 The following result on strength $3$ binary orthogonal arrays is well-known \cite{Colbourn:2006:HCD:1202540, rao1947factorial}. 
 \begin{lemma}
 If an OA$(m,n,2,3)$ exists, then $m\leq 2^{n-1}$. Moreover an OA$(2^{n-1},n,2,3)$ exists, the rows of which are all the binary vectors of length $n$ and odd weight.   
 \label{3strong}
 \end{lemma}

Let $C$ and $R$ be orthogonal arrays each of degree $k$ with $q$ levels with the zero vector in the first row.  
We define $C\boxplus R$ to be the array such that row $i$ and column $j$
contains the vector sum, calculated in ${\mathbb F}_q$, of the $i$th row of $C$ and the $j$th row of $R$. In turn, we call an array $L$ of type $I_k(m,n,q,t)$ {\em abelian} if and only if there exists $C$ and $R$ such that 
$L=C \boxplus R$, where 
 $C$ is an OA$(m,k,q,t)$ and $R$ is an OA$(n,k,q,t)$.
 (Here we use $C$ and $R$ to remind the reader that the first {\it C}olumn and first {\it R}ow of $L$ are, respectively, the orthogonal arrays $C$ and $R$.) 
If the replication is $1$, then such an array is abelian if and only if it is the subarray of the addition table for ${\mathbb F}_q^k$. 
Most constructions in this paper are abelian, however Section 5 contains 
some non-abelian constructions. 

In Section 2, we give some general recursive constructions that apply to all row-column factorial designs. In Section 3 we focus on the abelian case, 
using linear algebra to show that row-column factorial designs can be constructed from orthogonal arrays and matrices with certain independence properties.  
These are applied in Section 4 where we consider the strength $2$ case with an arbitrary number of levels. We solve this case completely when 
the number of levels $q$ is a prime power and 
the dimensions of the array are each a power of $q$; see Theorem \ref{biggerthanelvis}. This generalizes the binary case solved in  \cite{godolphin2019construction}. 

In Section 5 we find necessary and sufficient conditions for the existence of $I_{k}(4m,n,2,2)$ whenever $m\leq 5$; or $m$ is odd assuming the truth of  Conjecture \ref{whoknows}.  We 
also show that  $I_{k+\alpha}(2^{\alpha}b,2^k,2,2)$ exists whenever
$\alpha\geq 2$ and $2^{\alpha}+\alpha+1\leq k<2^{\alpha}b-\alpha$, assuming there exists a Hadamard matrix of order $4b$ (Theorem \ref{biggercases}). 
  Finally in Section 6 we consider the strength $3$ binary case, solving 
this whenever the dimensions are powers of $2$ (Theorem \ref{strength3}).

\section{General results}
 
 In this section we list some general observations and results that can be applied to row-column factorial designs of any strength. 
 
We start with some straightforward lemmas. 

\begin{lemma}
If $D$ is an array of type $I_k(m,n,q,t)$ then: 
\begin{itemize}
\item  $D$ is also an array of type $I_k(m,n,q,t')$
 for each $t'$ such that $1\leq t'\leq t$;
\item there exists an array of type $I_{k'}(m,n,q,t')$ for each $k'$
such that $1\leq k'\leq k$. 
\end{itemize}
\label{subarrays}
\end{lemma}   

\begin{lemma}
If there exist arrays of type $I_k(m,n,q,t)$ and $I_k(m',n,q,t)$ there exists an array of type $I_k(m+m',n,q,t)$. 
If there exist arrays of type $I_k(m,n,q,t)$ and $I_k(m,n',q,t)$ there exists an array of type $I_k(m,n+n',q,t)$.
\label{lem:glueing}
\end{lemma}

	The proof of the following lemma is
	a Kronecker product construction based on a similar construction for orthogonal arrays (Theorem III.7.20 from \cite{Colbourn:2006:HCD:1202540}, originally \cite{bushka}).

	\begin{lemma}
		\sloppy If there exist arrays of type $I_k(m,n,q,t)$ and $I_k(m',n',q',t)$ then 
		there exists an array of type $I_k(mm',nn',qq',t)$. 
		\label{products} 
	\end{lemma}
	
	\begin{proof}
	
		
		
		\sloppy 
		
		Let $D$ and $D'$ be arrays of type $I_k(m,n,q,t)$ and $I_k(m',n',q',t)$, respectively. 
	We construct an $mm'\times nn'$ array $ D \boxtimes D'$ as follows. 
	For each $(i,j)\in [mm']\times [nn']$,  write 
	$i=xm'+x'$ and $j=yn'+y'$ where $x\in [m]$, $x'\in [m']$, 
	$y\in [n]$, $y'\in [n']$, noting that the choices of $x$, $x'$, $y$ and $y'$  are unique and depend on $i$ and $j$. 
	In cell $(i,j)$ we place the vector $q'D(x,y)+D'(x',y')$, where $D(x,y)$ and $D'(x',y')$ are the vectors in cells $(x,y)$ and $(x',y')$ of $D$ and $D'$, respectively. 
	
	We next verify that $D\boxtimes D'$ is an array of type $I_k(mm',nn',qq',t)$.
	Fix a set $T$ of $t$ coordinates in column $j$ of $ D \boxtimes D' $ and let $(v_1,v_2,\dots ,v_t)\in [qq']^t$. As above, write 
	$j=yn'+y'$ for unique	$y\in [n]$, $y'\in [n']$. 
	For each $\alpha\in [t]$, let $x_{\alpha}\in [q]$ and $x_{\alpha}'\in [q']$ be unique solutions to $v_{\alpha}=x_{\alpha}q'+x_{\alpha}'$. 
	
	Since $D$ is of strength $t$, the vector $(x_1,x_2,\dots ,x_t)$ appears $n/q^t$ times in column $y$ of $D$ in the set of positions $T$. Similarly, the vector $(x_1',x_2',\dots ,x_t')$ appears $n'/(q')^t$ times in column $y'$ of $D'$ in the same set of positions $T$. 
	Thus $(v_1,v_2,\dots ,v_t)$ appears precisely $nn'/(qq')^t$ times in column 
	$j$ of $D\boxtimes D'$. 
By the same argument in transpose, each vector in $[qq']^t$ appears 
$mm'/(qq')^t$ in each row of 
 $D\boxtimes D'$.
	
	It remains to show that each vector in $[qq']^k$ appears the same number of times in the array 
	 $D\boxtimes D'$. The idea is similar to above. 
	 Let $(v_1,v_2,\dots ,v_k)\in [qq']^k$. 
	For each $\alpha\in [k]$, let $x_{\alpha}\in [q]$ and $x_{\alpha}'\in [q']$ be unique solutions to $v_{\alpha}=x_{\alpha}q'+x_{\alpha}'$. 
	 By the parameters of $D$ and $D'$, 
	 the vectors $(x_1,x_2,\dots ,x_k)$
	 and $(x_1',x_2',\dots ,x_k')$  
	 appear $mn/q^k$ and 
	 $m'n'/(q')^k$ times,  respectively, in the arrays $D$ and $D'$. 
	 Thus $(v_1,v_2,\dots ,v_k)$ appears precisely $mm'nn'/(qq')^k$ times in the array 
	 $D\boxtimes D'$. 
	\end{proof}

	
	 The $m\times n$ matrix of ${\bf 0}$ vectors of dimension $k$ is trivially an array $I_k(m,n,1,t)$ for any $1\leq k\leq t$. The following corollary is then immediate. 
	
	\begin{corollary}
	    \sloppy 
	    If there exists an array $I_k(m,n,q,t)$, then there exists an array $I_k(mm',nn',q,t)$ for any integers $m',n'\geq 1$. 
	    \label{cor:blowup}
	\end{corollary}

\section{Abelian row-column factorial designs}

In this section we find necessary and sufficient conditions on 
orthogonal arrays $C$ and $R$ such that 
$C\boxplus R$ is a row-column factorial design of strength $t$, where at least one of $C$ or $R$ is a vector space.

Since every row (column) in $C\boxplus R$ is equivalent to the first row (respectively, column), we have the following observation. 

\begin{lemma}
Let $C$ be an OA$(m,k,q,t)$ and let $R$ be an OA$(n,k,q,t)$. Then 
$L=C\boxplus R$ is an array of type $I_k(m,n,q,t)$ 
if and only if $L$ is a row-column factorial design, that is, the set of entries of the cells of $L$ is the $\lambda \times [q]^k$ factorial design.
\label{justcheckthefirst}
\end{lemma}

We next consider the extreme case when every row is a factorial design.  

	\begin{lemma} 
	If there exists an orthogonal array OA$(m,k,q,t)$  
	then there exists an array of type $I_k(m,q^k,q,t)$.
 \label{letsgettrivial}
	\end{lemma}
 
 	\begin{proof}
\sloppy 
Let $R$ be an OA$(q^k,k,q,t)$ where $t\leq k$. That is, the row vectors of $R$ are the  factorial design $[q]^k$.    
 Let $C$ be an OA$(m,k,q,t)$.  
Since the entries in each row of $C \boxplus R$ are trivially $[q]^k$, by the previous lemma 
 $C \boxplus R$ is an array of type $I_k(m,q^k,q,t)$.
	\end{proof}

In the following, given a matrix $A$ over the field $ \mathbb{F}_q $, let $<A>$ be a matrix whose row vectors are the rowspace of $A$; that is the vector space generated by the row vectors of $A$. 
An orthogonal array is called {\em linear} if its rows are the rows of a vector space (assuming without loss of generality that the zero vector is one of the rows). 
The theory of linear orthogonal arrays is closely related to that of linear codes; see for example Section 4.3 of 
\cite{hedayat1999orthogonal}. 

The following lemma is in some ways a standard result in this area, stemming from the original result by Bose (1961) \cite{bose1961some} which links the two theories. This lemma is also implied by Theorem 3.27 and Theorem 3.29 in \cite{hedayat1999orthogonal}.
We include a proof for thoroughness.  

Henceforth we say that a set $S$ of vectors is {\em $t$-independent} if every $t$-subset of $S$ is linearly independent \cite{damelin2007cardinality}. 

\begin{lemma}
Let $A$ be an $m\times n$ matrix of rank $m$ over the field ${\mathbb F}_q$ where $m\leq n$.    
Let $2\leq t\leq m$. Then $<A>$ is an OA$(q^m,n,q,t)$ if and only if
the set of columns of $A$ is $t$-independent. 
\label{fundam}
\end{lemma}

\begin{proof}
Let $C$ be a subset of $t$ distinct column vectors of $A$.
Let $B$ be the $m\times t$ sub-matrix which is $A$ restricted to these columns. 

If $C$ is a dependent set, there exists a non-zero vector ${\bf w}$ such that $B{\bf w}={\bf 0}$. 
Then, by the fundamental theorem of linear algebra, the vector ${\bf w}$ does not occur in $<B>$. 
In turn, within the set of vectors of $<A>$, within the $t$ positions determined by the columns of $C$,  the ordered sequence ${\bf w}$ 
does not occur. Thus $<A>$ does not have strength $t$.

Conversely, suppose that 
$C$ is an independent set. 
Let ${\bf w}\in {\mathbb F}_q^t$. 
Since $B$ has rank $t$, there are 
$q^{m-t}$ vectors ${\bf v}$ such that 
$B{\bf v}={\bf w}$. 
In turn, within the set of vectors of $<A>$, within the $t$ positions determined by the columns of $C$, each element of $({\mathbb F}_q)^t$ occurs $q^{m-t}$ times.
\end{proof}

\begin{corollary}
Let $K$ be a binary $m\times (n-m)$ matrix.  
Then $<[I|K]>$ is an OA$(2^m,n,2,2)$ 
 if and only if the columns of $K$ are distinct {\em and} each column of $K$ contains at least $2$ non-zero elements.
 \label{conditionsonk}
\end{corollary}

\begin{example}
From the previous corollary, $<A>$ is an OA$(2^6,8,2,2)$:   
\begin{equation*}
		A = \left( \begin{array}{cccccc:cc}
			1 & 0 & 0 & 0 & 0 & 0 & 1 & 0 \\
 		    0 & 1 & 0 & 0 & 0 & 0 & 0 & 1 \\
 		    0 & 0 & 1 & 0 & 0 & 0 & 1 & 1 \\
 		    0 & 0 & 0 & 1 & 0 & 0 & 1 & 0 \\
 		    0 & 0 & 0 & 0 & 1 & 0 & 0 & 1 \\ 
 		    0 & 0 & 0 & 0 & 0 & 1 & 1 & 1 \\
 		 \end{array}\right) 
	\end{equation*}
\label{eggsone}
\end{example}

\begin{theorem} 
    Let $G$ be an $m\times k$ matrix  and let 
    $A$ be an $N\times k$ matrix of full rank, each over the field ${\mathbb F}_q$ where $N\leq k$. 
    Let $ A^{\perp} $ be a $k\times (k-N)$ matrix the columns of which are a basis for the nullspace of $A$.
    Then $G \ \boxplus <A>$ is an $m \times q^N$ row-column factorial design if and only if $GA^\perp$ is an OA$(m,k-N,q,k-N)$. 
\label{bigdeal1}
\end{theorem}

\begin{proof}
\sloppy 
Let $M=G \ \boxplus <A>$ and let 
$ { \bf v_i} \in\mathbb{F}_q^k $ denote the $i^{th}$ row vector of $G$. Then the elements of the 
$i^{th}$ row of $M$ form the coset ${ \bf v_i} + <A>$ of $<A>$ in $\mathbb{F}_q^k$. Note that $<A>$ has exactly $q^{k-N}$ distinct cosets in $\mathbb{F}_q^k$. We show that each of these cosets appears  the same number of times as sets of entries in a row of $M$. 
To this end, observe that for $1 \leq i,j \leq m$:
    $$ { \bf v_i} + <A> = { \bf v_j} + <A>  \iff { \bf v_i} - { \bf v_j} \in <A> \iff { \bf v_i} A^\perp = { \bf v_j} A^\perp.$$ Thus the set of  entries in two rows of $M$ are identical if and only if the corresponding rows in $GA^\perp$ are identical. 
    Thus $M$ is a row-column factorial design (that is, the set of entries of $M$ form a factorial design) if and only if each element of 
    $(\mathbb{F}_q)^{k-N}$ occurs the same number of times as a row of $GA^\perp$. 
    In turn, this is true if and only if  $GA^\perp$ is an OA$(m,k-N,q,k-N)$. 
\end{proof}

\begin{theorem}  
    Let $G$ be an OA$(m,k,q,t)$ 
    and let $<A>$ be an OA$(q^N,k,q,t)$ 
    where $A$ is an $N\times k$ matrix of full rank and $N\leq k$. 
    Let $A^\perp$ be a $k\times (k-N)$ matrix whose columns generate the nullspace 
    of $A$. 
 Suppose that $GA^\perp$ is an OA$(m,k-N,q,k-N)$. 
    Then $G \ \boxplus <A>$ is an array of type $I_k(m,q^N,q,t)$.
\label{bigdeal}
\end{theorem}

\begin{proof}
\sloppy 
    Let $M = G \ \boxplus <A>$. 
     The result follows from Theorem \ref{bigdeal1} and Lemma \ref{justcheckthefirst}.   
    \end{proof}

In the next example (and in Section 5) we make use of the result (well-known to coding theorists) that over any 
field, the nullspace of the matrix $[I|K]$ is equal to the columnspace of the matrix $[-K^T|I]^T$ (Remark 1.5, \cite[p.~677]{Colbourn:2006:HCD:1202540}).

	\begin{example} \label{eggstwo}
	We continue with Example \ref{eggsone} to show  that $G \ \boxplus <A>$ is an $I_{8}(12,2^6,2,2)$,  using Theorem \ref{bigdeal} the following $G= \textup{OA}(12,8,2,2)$. 

	\begin{equation*}
    A^\perp = \left( \begin{array}{cc}
              
			 1 & 0 \\
 		     0 & 1 \\
 		     1 & 1 \\
 		     1 & 0 \\ 
 		     0 & 1 \\ 
 		     1 & 1 \\\hdashline
 		     1 & 0 \\
             0 & 1 \\
 		 \end{array}\right) 
 		 \quad
		G = \left( \begin{array}{cccccccc}
			0 & 0 & 0 & 0 & 0 & 0 & 0 & 0 \\
 		    0 & 0 & 0 & 1 & 1 & 1 & 0 & 1 \\
 		    1 & 0 & 0 & 0 & 1 & 1 & 1 & 0 \\
 		    0 & 1 & 0 & 0 & 0 & 1 & 1 & 1 \\
 		    1 & 0 & 1 & 0 & 0 & 0 & 1 & 1 \\ 
 		    1 & 1 & 0 & 1 & 0 & 0 & 0 & 1 \\
 		    0 & 1 & 1 & 0 & 1 & 0 & 0 & 0 \\ 
 		    1 & 0 & 1 & 1 & 0 & 1 & 0 & 0 \\
 		    1 & 1 & 0 & 1 & 1 & 0 & 1 & 0 \\
 		    1 & 1 & 1 & 0 & 1 & 1 & 0 & 1 \\
 		    0 & 1 & 1 & 1 & 0 & 1 & 1 & 0 \\
 		    0 & 0 & 1 & 1 & 1 & 0 & 1 & 1 
		\end{array}\right) 
		\quad 
		GA^\perp = \left( \begin{array}{cc}
			0 & 0 \\
			0 & 1 \\ 
			1 & 0 \\
			0 & 1 \\
			1 & 0 \\
			0 & 0 \\
			1 & 1 \\
			0 & 0 \\
			1 & 0 \\
			1 & 1 \\
			0 & 1 \\
			1 & 1 \\
 		\end{array}\right)
		\end{equation*}
	\end{example}

	The following is a generalization of Theorem 5 in \cite{rahim2021row-column}.
	
	\begin{theorem}
	Let $q$ be a prime power. 
		Let $M,N\geq 1$ and $k \leq M+N$. 
		Suppose there exists a $k\times M$ matrix $A$ and a $k\times N$ matrix $B$ such that (a) the $k\times (M+N)$ matrix $[A|B]$ has rank $k$; (b) the rows of $A$ are a $t$-independent set of vectors; and  (c) the rows of $B$ are a $t$-independent set of vectors. 
	Then there exists an abelian array of type 
		$I_{k}(q^M, q^N, q,t)$.  	
		\label{thm:FR.Polynomilas.existence.}
			\end{theorem}

	\begin{proof}
\sloppy 
Assuming the conditions of the theorem,  the rows of $[A|B]$ are a linearly independent set of $ k $ vectors:
		\begin{equation} \label{eq:vectors}
		(a_{r,0}+ a_{r,1}+ \cdots + a_{r,M+N-1}); \ r \in [k]
		\end{equation}
		in the vector space $ \mathbb{F}_q^{M+N} $ 
		such that:

		\begin{enumerate}
			\item[(i)]  The set of vectors $\{(a_{r,0}, ... , a_{r,M-1}): r\in [k]\}$ is $t$-independent; and   
			\item[(ii)]  the set of vectors $\{(a_{r,M}, ... , a_{r,M+N-1}): r\in [k]\}$ is $t$-independent.
		\end{enumerate}

Corresponding to each vector in (\ref{eq:vectors}) we construct a $q^{M} \times q^{N} $ array $A_r$ by using a polynomial $f_r$, where

\begin{equation*}
	f_r(x_0, ... , x_{M+N-1}) = a_{r,0}x_{0}+ \cdots + a_{r,M+N-1}x_{M+N-1}.
\end{equation*}

Label the rows and columns of $A_r$ by using the set of all $M$-tuples  and $N$-tuples, respectively, over the field  $ \mathbb{F}_q $. We place the element $ f(b_0,...,b_{M-1},c_0, ..., c_{N-1}) $ in the intersection of row $ (b_0,..., b_{M-1}) $ and column $ (c_0, ... , c_{N-1}) $ of the array $A_r$. 

We next form an array $D$ by overlapping the arrays $A_r$, $r\in [k]$. That is, cell $(i,j)$ of $D$ contains a vector of dimension $k$ the $r${th} coordinate of which is the entry of cell $(i,j)$ of $A_r$. 
We claim that the array $D$ is an $I_{k}(q^M, q^N, q,t)$.

We first show that $D$ is an array of strength $ t $. Let $ T $ be a subset of $ [k] $ of size $ t $ and consider a sequence $(\alpha_1, \dots, \alpha_t) $ in $ [q]^t $. For a fixed column in $ D $, the system of equations
\begin{equation*}
	f_r(x_0, \dots , x_{M+N-1} ) = \alpha_r; \hspace{1cm} r \in T
\end{equation*}
reduces to:
\begin{equation*}
	 a_{r,0}x_0+ a_{r,1}x_1+ \cdots + a_{r,M-1}x_{M-1} = \alpha_r + K_r; \hspace{1cm} r \in T
\end{equation*}
where $ K_r $ is a constant in $ \mathbb{F}_q $ for each $ r \in T $.

By condition (i), the above system with $ M $ variables has rank $ t $. 
Therefore it has exactly $ q^{M-t} $ solutions in $ \mathbb{F}_q $. Thus each column of $ D $ is an orthogonal array of type OA$(q^M,k,q,t) $. Similarly, we can show that each row of $ D $ also forms an orthogonal array of type OA$(q^N,k,q, t) $. Hence the strength $ t $ conditions is satisfied. 

Now to show that $D$ is a $q^k$-full factorial design. Consider a sequence $(\alpha_0, \alpha_1, \dots , \alpha_{k-1})$ in $\mathbb{F}_q^k$. Since the system of equations:
\begin{equation*}
	 a_{r,0}x_0+ a_{r,1}x_1+ \cdots + a_{r,M+N-1}x_{M+N-1} = \alpha_r, \hspace{1cm} r \in [k],
\end{equation*}
has $M+N$ variables and rank $k$, it has exactly $q^{M+N-k}$ solutions in $\mathbb{F}_q$. Thus each sequence in $\mathbb{F}_q^k$ appears exactly $q^{M+N-k}$ times in $D$.
\end{proof}

\section{Strength $2$ with arbitrary number of levels}

In this section we consider row-column factorial designs of the form $I_k(q^M,q^N,q,2)$. 
The aim of this section is to prove the following theorem.
\begin{theorem}
Let $2\leq M\leq N$, let $q$ be a prime power and let $k\geq 2$.  
Then there exists an array of type $I_k(q^M,q^N,q,2)$ if and only if $k\leq M+N$, 
$k\leq (q^M-1)/(q-1)$ and $(k,M,q)\neq (3,2,2)$. 
\label{biggerthanelvis}
\end{theorem}

Lemma \ref{products} and the previous theorem imply the following corollary. 

\begin{corollary}
Let $q_1\leq q_2\leq \dots \leq q_{\alpha}$ be powers of distinct primes and $q_1\neq 4$. 
Let $q=q_1q_2\dots q_{\alpha}$, $k\leq (q_1^M-1)/(q-1)$, $2\leq M\leq N$ and $2\leq k\leq M+N$. 
Then there exists an array of type $I_k(q^M,q^N,q,2)$. 
\end{corollary}

The elements of $[2]^2$ form the rows of an OA$(4,2,2,2)$. In turn, 
Lemma \ref{letsgettrivial} implies the existence of $I_2(4,4,2,2)$. 
This observation, together with the following two lemmas and Theorem \ref{thm:FR.Polynomilas.existence.}, imply the above theorem. 
			
\begin{lemma}
	Let $ N\geq M \geq 2$ be integers and $q\geq 2$ be a prime power, with $ M+N \leq (q^M-1)/(q-1)$ and $ (M,q)\neq (2,2) $. 
		Then there exists an $(M+N)\times M$ matrix $A$ and an $(M+N)\times N$ matrix $B$ such that (a) $[A|B]$ has full rank; (b) 
		no two rows of $A$ are parallel; and (c) no two rows of $B$ are parallel. 
\end{lemma}	
	
\begin{proof}
We split the proof in different cases. In each case we describe a square matrix $L =[A|B]$ with the required properties.  

\textbf{Case I:} \hspace{0.5cm} When $ M = 3 $ and $ q = 2 $.

In this case, $3 \leq N \leq 4$. For $N = 4$ we define the matrix $ L $ to be,

\begin{equation} \label{eq:matrix.L.particular}
	L = \left(\begin{array}{ccc|ccc:c}
		1 & 0 & 0 & 0 & 1 & 1 & 0 \\
		0 & 1 & 0 & 1 & 1 & 0 & 0 \\
		0 & 0 & 1 & 1 & 1 & 1 & 0 \\
		1 & 1 & 0 & 1 & 0 & 0 & 0 \\
		0 & 1 & 1 & 0 & 1 & 0 & 0 \\
		1 & 0 & 1 & 0 & 0 & 1 & 0 \\ \hdashline
		1 & 1 & 1 & 0 & 0 & 0 & 1 \\ 		
	\end{array} \right)
\end{equation}

\vspace{2mm}
The above matrix has full rank over $ \mathbb{F}_2 $ and also satisfies the conditions (b) and (c). Now for the case $ N=3 $, we can take the $ 6 \times 6 $ sub-matrix (as shown in (\ref{eq:matrix.L.particular})) of the above matrix obtained by deleting the final row and column. Observe that this matrix also has full rank and satisfies the conditions in the lemma.
This completes Case I. 

For all other cases, we take  $ L $ to be of the form:

\begin{equation} \label{eq:matrixx.L}
		L = \left( \renewcommand\arraystretch{1.8}\begin{tabular}{>{$} P{2.0cm} <{$} | >{$} P{2.0cm} <{$} : >{$} P{2.0cm} <{$} }
			I_M		 & I_M & \textbf{O} \\ \hdashline
			C_M-I_M  & C_M & \textbf{O} \\ \hdashline 
			S & \textbf{O} & I_{N-M} \\
		\end{tabular} \right)
\end{equation}

\vspace{5mm}

Where $ I_M $ is an identity matrix of order $ M $ and 
$ \textbf{O} $ is a matrix of zeroes of appropriate size. Matrices $ C_M $ and $ S $ are to be defined later.

Label the columns of the matrix $ L $ in (\ref{eq:matrixx.L}) by $ c_i, \ i \in [M+N] $. Note that for any choices of matrices $ C_M $ and $ S $, the column operations;
\begin{equation}
	 c_{M+i} - c_i \longrightarrow c_{M+i}, \hspace{3mm} \text{for each} \hspace{3mm} i \in [M], 
 \end{equation}
   transfers the matrix $ L $ into a lower triangular matrix with entry $ 1 $ on the main diagonal. 
Thus condition (a) is satisfied for any choice of $C_m$ and $S$. 

\textbf{Case II:} \hspace{0.5cm} When $ M \geq 4 $.

 \sloppy In this case let $ C_M $ be a $M\times M$ matrix with exactly one $0$ in each row and column, $1$'s on the main diagonal and $1$'s in every other cell. Observe that condition (c) is satisfied. 
 Let $ D $ be the set of all rows in $ I_M $ and $ C_M - I_M$. It is easy to see that no two vectors in $ D $ are parallel. Let $ W $ be the largest set of non-parallel vectors in $ \mathbb{F}_q^M $ containing $ D $. Then  $ \vert W \vert = (q^M-1)/(q-1)$. We define $ S $ to be the matrix for which each row is a distinct element in $ W \setminus D $. 
 Thus condition (b) is satisfied. This completes Case II. 

In the remaining cases the matrix $ S $ can be obtained in the similar manner from the corresponding  $ C_M $, again satisfying condition (b).

\textbf{Case III:} \hspace{0.5cm} When $ 2 \leq M \leq 3 $ and $q$ is odd.

 In this case we define the matrices $ C_M, M\in \{2,3\} $ to be: 
 
\begin{equation*}
C_2 = \left(\begin{array}{cc}
	2 & 2 \\
	1 & 2 \\	
	\end{array} \right) \hspace{5mm} \text{and} 
\hspace{5mm} 
C_3 = \left(\begin{array}{ccc}
	2 & 0 & 1 \\
	1 & 2 & 0 \\	
	0 & 1 & 2 \\
\end{array} \right).
\end{equation*}

\vspace{2mm}
 \textbf{Case IV:} \hspace{0.5cm} When $ M = 3 $ and $ q = 2^l, \ l \geq 2 $.
 
 In this case we define $C_3 $ as follows. 
 
 \begin{equation*}
 	C_3 = \left(\begin{array}{ccc}
 		\alpha + 1 & 1 & 1 \\
 		1 & \alpha + 1 & 1 \\	
 		1 & 1 & \alpha + 1 \\
 	\end{array} \right),
 \end{equation*}

\vspace{2mm}
 where $ \alpha $ is a primitive element of the field $ \mathbb{F}_q $.
\end{proof}

	\begin{lemma}
		If $b$ is odd, there does not exist an array  of type $I_{3}(4, 4b, 2,2)$. 
		\label{thm:exception1}
	\end{lemma}	
\begin{proof}
Suppose that an array $D$ exists of type $I_{3}(4, 4b, 2,2)$. 
 Then each column of $D$ is an OA$(4,3,2,2)$. By inspection, 
the vectors in any column of $D$ are either all the vectors of even weight or all 
the vectors of odd weight; we refer to these columns as type A or B, respectively. 
	
	\begin{table}[H]
	\centering
	\begin{tabular}{|c|}
		\multicolumn{1}{c}{A} \\ \hline
	000 \\
	101 \\
	011 \\
	110 \\	\hline
	\end{tabular}
	\begin{tabular}{ccc}
	 &&\\
	 &&\\
	 &&\\
	 &&\\	
	\end{tabular}
	\begin{tabular}{|c|}
	\multicolumn{1}{c}{B} \\	\hline
	001 \\
	100 \\
	010 \\
	111 \\	\hline
	\end{tabular}
	\end{table}

Since the vectors in $ A $ and $ B $ form a partition of $ \mathbb{F}_2^3 $ and the entries of $D$ form a factorial design, $D$ must contain exactly $ 2b $ columns of each type. 

Now consider a row $ R $ in $D$; without any loss of generality we may assume that the first two coordinates of $ R $ have the following form:

$$ 
\overbrace{00 \ \ \ 00 \ \ \ \dots  \ \ \ 00  }^{B1} \ | \ \overbrace{01 \ \ \ 01 \ \ \ \dots  \ \ \ 01  }^{B2} \ | \ \overbrace{10 \ \ \ 10 \ \ \ \dots  \ \ \ 10  }^{B3} \ | \ \overbrace{11 \ \ \ 11 \ \ \ \dots  \ \ \ 11  }^{B4}
$$

where each $ Bi$ has size $b$. 
Let $x$ be the number of zeros at the third coordinate in $B1$, then without loss of generality $x \geq (b+1)/2$. By strength two property, the number of zeros in the third coordinate in $B2, B3$ and $B4$ is $ b-x, b-x$ and $x$ respectively. This implies that there are $x$ vectors of type A in each $Bi$. Consequently, $R$ contains $4x \geq 2b+2$ vectors of type A. This is a contradiction since $D$ contains exactly $2b$ columns of each type.
\end{proof}

\section{Binary row-column factorial designs of strength $2$} 

In this section we restrict ourselves to the binary case. 
We exploit the theory developed in Section 3 to give existence results for arrays of the form $I_k(4m,n,2,2)$. We focus on the case where $m$ is odd, however the next theorem is also true when $m$ is even. 
The main results in this section are given in 
 Theorems \ref{biggercases}, \ref{thisisalabel} and \ref{section5main}.  
	
\begin{theorem}
Let $k\geq 5$. 
Let $m\geq 3$ be odd and suppose there exists 
an OA$(4m,k,2,2)$ with 
two subsets of column vectors $V$ and $W$ such that:
\begin{itemize}
\item $|V|,|W|\geq 3$;
\item there exists ${\bf v}\in V\setminus W$ and ${\bf w}\in W\setminus V$ such that 
$V\setminus \{{\bf v}\}\neq W\setminus \{{\bf w}\}$; 
\item $(\sum_{{\bf x}\in V} {\bf x})$ is orthogonal to 
$(\sum_{{\bf y}\in W} {\bf y})$. 
\end{itemize}
 Then
there exists an abelian  $I_k(4m,2^{k-2},2,2)$. 
\label{eightisenough}
\end{theorem}

\begin{proof}
Let $G = \textup{OA}(4m,k,2,2)$ be an orthogonal array satisfying the conditions of the theorem. Let 
$G=[{\bf v}_1|{\bf v}_2|\dots |{\bf v}_k]$.  
Without loss of generality, assume that 
${\bf v}_{k-1}\in V\setminus W$ and 
${\bf v}_{k}\in W\setminus V$. 
Define 
a $(k-2)\times 2$ matrix $K$ over ${\mathbb F}_2$ such that 
the first column of $K$ contains a $1$ in the $j$th row if and only if ${\bf v}_{j}\in V$. Similarly, 
the second column of $K$ contains a $1$ in the $j$th row if and only if ${\bf v}_{j}\in W$. Observe furthermore that 
$[K^T|I]^T$ has the same property. 

Next, let $A$ be the $(k-2)\times k$ matrix defined by $A=[I|K]$. 
If the columns of $K$ are identical, 
then 
$V\setminus \{{\bf v}\}= W\setminus \{{\bf w}\}$, a contradiction.
Moreover, since 
$|V|,|W|\geq 3$, 
the columns of $K$ each have at least two $1$'s.  
Thus, by Corollary \ref{conditionsonk}, 
$<A>$ is an OA$(2^{k-2},k,2,2)$.

Define $A^\perp=[K^T|I]^T$. Observe that
$GA^\perp$ is a $4m \times 2$ matrix with columns
given by $\sum_{{\bf x}\in V} {\bf x}$ and  
$\sum_{{\bf y}\in W} {\bf y}$. 
 By definition, $GA^\perp$ is an 
OA$(4m,2,2,2)$. 
Thus, by Theorem \ref{bigdeal}, 
$G \ \boxplus <A>$
is an $I_k(4m,2^{k-2},2,2)$. 
\end{proof}

Now, observe that the matrix $G$ from Example \ref{eggstwo} is an OA$(12,8,2,2)$ with the property that 
${\bf v}_1+{\bf v}_3+{\bf v}_4+{\bf v}_6+{\bf v}_7$ is orthogonal to  
${\bf v}_2 +{\bf v}_3+{\bf v}_5+{\bf v}_6+{\bf v}_8$.
Moreover, $G$ embeds in the Hadamard matrix $H(12)$ of order 12 ({\tt had.12}, \cite{sloane1999library}). 
 Thus we have the following corollary. 

\begin{corollary}
There exists an abelian  $I_k(12,2^{k-2},2,2)$ where $8\leq k\leq 11$.     
\end{corollary}

\begin{corollary} 
There exists an abelian  $I_k(20,2^{k-2},2,2)$ where $8\leq k\leq 19$.     
\end{corollary}

\begin{proof}
The following is a transpose of an OA$(20,8, 2,2)$ which has the property that
${\bf v}_1+{\bf v}_3+{\bf v}_4+{\bf v}_6+{\bf v}_7$ is orthogonal to  
${\bf v}_2 +{\bf v}_3+{\bf v}_5+{\bf v}_6+{\bf v}_8$.
The result follows by Theorem \ref{eightisenough}.

$$ \left[
\begin{array}{cccc cccc cccc cccc cccc}
1&	1&	1&	1&	0&	1&	0&	0&	0&	0&	1&	0&	0&	0&	0&	1&	1&	1&	1&	0 \\
1&	1&	0&	0&	0&	0&	1&	0&	1&	1&	1&	1&	0&	0&	1&	0&	1&	1&	0&	0 \\
1&	1&	1&	1&	1&	1&	1&	1&	1&	1&	0&	0&	0&	0&	0&	0&	0&	0&	0&	0 \\
1&	1&	0&	0&	1&	1&	0&	0&	1&	0&	0&	1&	0&	1&	0&	1&	0&	1&	0&	1 \\
1&	1&	0&	0&	1&	1&	0&	0&	0&	1&	0&	0&	1&	0&	1&	0&	1&	0&	1&	1 \\
1&	1&	1&	0&	0&	0&	1&	1&	0&	0&	0&	1&	1&	0&	0&	1&	1&	0&	0&	1 \\
1&	1&	1&	1&	1&	0&	0&	0&	0&	0&	1&	1&	1&	1&	1&	0&	0&	0&	0&	0 \\
1&	1&	0&	1&	0&	0&	1&	1&	0&	0&	0&	0&	0&	1&	1&	0&	0&	1&	1&	1 \\
\end{array}
\right]
$$

The above array consists of columns 2 to 9 of the Hadamard matrix ({\tt had.20.toncheviv},  \cite{sloane1999library})  with the permutation $(2 \ \ 4)(5 \ \ 7)(3 \ \ 8 \ \ 6 \ \ 9)$ applied to its columns.
\end{proof}

\begin{corollary}
For any odd $m\geq 3$, there exists an abelian $I_k(4m,2^{k-2},2,2)$ where $8\leq k\leq 11$.     
\label{eightandabove}
\end{corollary}
	
\begin{proof}
By Theorem \ref{biggerthanelvis} there exists an $I_k(16,2^{k-2},2,2)$ where $8 \leq k \leq 15$. Thus the result follows by previous two corollaries and Lemma \ref{lem:glueing}.
\end{proof}

Via a counting argument, Theorem \ref{eightisenough} cannot work for $m$ odd if 
$k\leq 6$. We outline this argument in the conclusion in Lemma \ref{countingargument}.   
Moreover, computational results 
show that $k=7$ does not work in the cases $m\in \{3,5\}$.  

The above and Theorem \ref{eightisenough} thus motivate the following conjecture, which is stronger than the Hadamard conjecture. 
\begin{conjecture}
For each odd $m$, there exists a Hadamard matrix $4m$  which yields an orthogonal array OA$(4m,8,2,2)$ satisfying the conditions of Theorem \ref{eightisenough}.
\label{whoknows}
\end{conjecture}

If the above conjecture is true, then by Theorem \ref{eightisenough},  
there exists an 
$I_{k}(4m,2^{k-2},2,2)$ for any $8\leq k\leq 4m-1$. 

We next focus on a strategy for the case $k\leq 7$. 
 Our constructions are typically non-abelian. 
In the following, $\oplus$ is  a binary operation that gives the concatenation of two vectors.
 That is, 
 $$(a_1,a_2,\dots ,a_r)\oplus (b_1,b_2,\dots ,b_s)=(a_1,a_2,\dots ,a_r,b_1,b_2,\dots ,b_s).$$
The following lemma is implied by the definition of an orthogonal array. 
Note that ${\bf 1}$ is the vector containing only $1$'s.

\begin{lemma}
Consider a set $S$ of $2^k $ binary vectors of dimension $ k + 2$ with the following properties: 
\begin{itemize}
    \item For each ${\bf v}\in [2]^k$, the vector ${\bf v}\oplus (i,j)\in S$, for some $i,j\in [2]$;
    \item  For each $(i,j)\in [2]^2$, 
there are precisely $2^{k-2}$ vectors in $S$ of the form 
${\bf v}\oplus (i,j)$ for some ${\bf v}$; 
    \item The vector ${\bf v}\oplus (i,j)\in S$ if and only if $({\bf 1}+{\bf v})\oplus (i,j)\in S$.
\end{itemize}
Then the vectors of $S$ are the rows of an OA$(|S|,k+2,2,2)$.   
\label{strength2}
\end{lemma}

\begin{lemma}
Let $\{{\bf w}, {\bf x}, {\bf y}, {\bf z},{\bf 1}\}$ be a set of linearly independent binary vectors of 
dimension $k\geq 5$.
Consider the $4\times 2^k$ array of vectors of dimension $k+2$ given by $C \boxplus R$, 
where $ R $ is the set of $2^k$ vectors with $0$ in the final two positions and
$$C=({\bf w}\oplus (0,0), {\bf x}\oplus (0,1), {\bf y}\oplus (1,0), 
{\bf z}\oplus(1,1))^T.$$ 
Then the elements in each column of $C\boxplus R$ can be rearranged so that each row is an 
OA$(2^k,k+2,2,2)$. 
\end{lemma}	

\begin{proof}
Observe that $H$ is a $4\times 8$ subarray of $C\boxplus R$:
$$
H=\begin{array}{|c|c|c|c|c|c|c|c|}
\hline
{\bf w} & {\bf x} & {\bf y} & {\bf z} & {\bf t} & {\bf u} & {\bf s} & {\bf v} \\  
\hline
{\bf x} & {\bf w} & {\bf t} & {\bf u} & {\bf y} & {\bf z} & {\bf v} & {\bf s}
\\
\hline
{\bf y} & {\bf t} & {\bf w} & {\bf s} & {\bf x} & {\bf v} & {\bf z} & {\bf u}
\\
\hline
{\bf z} & {\bf u} & {\bf s} & {\bf w} & {\bf v} & {\bf x} & {\bf y} & {\bf t}
\\
\hline
\end{array},$$
where
${ \bf s}={ \bf w}+ { \bf y}+ { \bf z}$, 
${\bf t}={\bf w}+ {\bf x}+ {\bf y}$, 
${\bf u}={\bf w}+ {\bf x}+ {\bf z}$, 
${\bf v}={\bf x}+ {\bf y}+ {\bf z}$ and 
the vectors in the first, second, third and fourth rows are concatenated with $(0,0)$, $(0,1)$, $(1,0)$ and $(1,1)$, respectively. 

Next, let $H'$ be the $4\times 8$ array formed by replacing each vector ${\bf a}$ in $H$  with the vector ${\bf a}+({\bf 1}\oplus (0,0))$. 
We next arrange 
the entries in each column of $[H|H']$. 
We mark the elements of $H$ as follows:

$$
H=\begin{array}{|c|c|c|c|c|c|c|c|}
\hline
\underline{{\bf w}} & {\bf x}^{\circ} & {\bf y}^{\ast}  & {\bf z}  & {\bf t} & {\bf u}^{\ast}  & {\bf s}^{\circ}  & \underline{{\bf v}} \\  
\hline
{\bf x}^{\ast} & {\bf w} & \underline{{\bf t}} & {\bf u}^{\circ}  & {\bf y}^{\circ}  & \underline{{\bf z}} & {\bf v} & {\bf s}^{\ast} 
\\
\hline
{\bf y} & {\bf t}^{\ast} & {\bf w}^{\circ}  & \underline{{\bf s}} & \underline{{\bf x}} & {\bf v}^{\circ}  & {\bf z}^{\ast}  & {\bf u}
\\
\hline
{\bf z}^{\circ}  & \underline{{\bf u}} & {\bf s} & 
{\bf w}^{\ast}  & 
{\bf v}^{\ast}  & {\bf x} & \underline{{\bf y}} & {\bf t}^{\circ} 
\\
\hline
\end{array}.$$
Next, rearrange the elements in each column of $H$ so that elements with the same mark are in the same row, with a corresponding permutation applied to each column of $H'$. Let the resultant $4\times 16$ matrix be $J$. 

Now, replace each vector of the form 
${\bf a}$ in $J$
with the vector $({\bf v}\oplus(0,0)) + {\bf a}$  to obtain a $4\times 16$ matrix
$J'$. 
Observe that for each row of $K=[J|J']$ and for each ${\bf g}\in G=<{\bf w}, {\bf x}, {\bf y}, {\bf z},{\bf 1}>$, there exists $i$ and $j$ such that 
${\bf g}\oplus (i,j)$ is in that row. 
Moreover, each column from $K$ is a column from  
$C\boxplus R$ with elements permuted. 

Let $G$, ${\bf z}_0+G,\dots ,{\bf z}_{\alpha-1}+G$ be the cosets of $G$ in 
$({\mathbb F}_2)^k$, where $\alpha=2^{k-5}$.  
For each $i\in [\alpha]$, let $K_i$ be formed from 
$K$ by replacing each entry  ${\bf a}$ of  $K$
with $({\bf z}_i\oplus (0,0))+{\bf a}$. 

Then, observe that $[K_0|K_1|\dots |K_{\alpha}]$ can be formed from $C\boxplus R$ by permuting the elements in each column. 
Moreover, the resultant rows each now satisfy the conditions of Lemma \ref{strength2}. 
\end{proof}

\begin{example}
Let ${\bf w} = 10000, {\bf x} = 10111, {\bf y}= 01101,$ and $ {\bf z} = 01011$. Then $H$ and $H'$ in the proof of above lemma are as follows:

\begin{table}[H]
	\centering
			\renewcommand{\arraystretch}{1.3}
			$H = $ 
			\resizebox{12cm}{!}{%
			\begin{tabular}{c}
			$\begin{array}{|c|c|c|c| c|c|c|c|}
			    \hline
            	\underline{1000000}&	1011100^{\circ}&	0110100^{\ast}&	0101100&	0101000&	0110000^{\ast}&	1011000^{\circ}&	\underline{1000100}\\\hline
                1011101^{\ast}&	1000001&	\underline{0101001}&	0110001^{\circ}&	0110101^{\circ}&	\underline{0101101}&	1000101&	1011001^{\ast}\\\hline
                0110110&	0101010^{\ast}&	1000010^{\circ}&	\underline{1011010}&	\underline{1011110}&	1000110^{\circ}&	0101110^{\ast}&	0110010\\\hline
                0101111^{\circ}&	\underline{0110011}&	1011011&	1000011^{\ast}&	1000111^{\ast}&	1011111&	\underline{0110111}&	0101011^{\circ}\\\hline
            \end{array} $ 
            \end{tabular}%
			}
\end{table}

\begin{table}[H]
	\centering
			\renewcommand{\arraystretch}{1.3}
			$H' =$
			\resizebox{12cm}{!}{%
			\begin{tabular}{c}
			$\begin{array}{|c|c|c|c| c|c|c|c|}
			    \hline
            	\underline{0111100}&	0100000^{\circ}&	1001000^{\ast}&	1010000&	1010100&	1001100^{\ast}&	0100100^{\circ}&	\underline{0111000}\\\hline
                0100001^{\ast}&	0111101&	\underline{1010101}&	1001101^{\circ}&	1001001^{\circ}&	\underline{1010001}&	0111001&	0100101^{\ast}\\\hline
                1001010&	1010110^{\ast}&	0111110^{\circ}&	\underline{0100110}&	\underline{0100010}&	0111010^{\circ}&	1010010^{\ast}&	1001110\\\hline
                1010011^{\circ}&	\underline{1001111}&	0100111&	0111111^{\ast}&	0111011^{\ast}&	0100011&	\underline{1001011}&	1010111^{\circ}\\\hline 
            \end{array}$
            \end{tabular}%
			} .
\end{table}

\end{example}

\begin{corollary}
Let $G = \textup{OA}(4m,k+2,2,2)$ 
be an orthogonal array such that the rows partition into sets of $4$ vectors of the form 
$$\{{\bf w}\oplus (0,0), {\bf x}\oplus (0,1), {\bf y}\oplus (1,0), 
{\bf z}\oplus(1,1)\}$$  
where $\{{\bf w}, {\bf x}, {\bf y}, {\bf z},{\bf 1}\}$ is a linearly independent set. 
Let $R$ be the set of $2^k$ vectors with $0$ in the final two positions. 
Then the elements in each column of $G\boxplus R$ can be rearranged to create an $I_{k+2}(4m,2^k,2,2)$. 
\label{yesyes}
\end{corollary}
	
\begin{proof}
From the previous lemma it suffices to check that $G\boxplus R$ is a factorial design. Let $A$ be a $k \times (k+2) $ matrix of the form $ [I | {\bf 0}]$. Observe that $R = <A>$.
The nullspace of $A$ is generated by the columns of $A^\perp=[{\bf 0}|I]^T$. 
Thus the columns of $GA^\perp$ are the last two columns of $G$ which are by definition orthogonal. 
The result then follows from 
 Theorem $\ref{bigdeal1}$. 
\end{proof}	
	
	\begin{corollary}
There exists $I_7(12,32,2,2)$. 
\label{yesyes2}
\end{corollary}

\begin{proof}
We present an orthogonal array of type OA$(12,7,2,2)$ in Table \ref{tabbleOA12}, that satisfies the conditions of Corollary \ref{yesyes}.
The dashed lines partition the rows into three sets of the form
$\{{\bf w}\oplus (0,0), {\bf x}\oplus (0,1), {\bf y}\oplus (1,0), 
{\bf z}\oplus(1,1)\}$ such that in each case  
$ \{ {\bf w}, {\bf x}, {\bf y}, {\bf z}, {\bf 1} \} $
is linearly independent. 

\begin{table}[H]
	\centering
			\renewcommand{\arraystretch}{1.2}
			$$ \left( \begin{array}{ccccc : cc}
            	1 & 0 & 0 & 0 & 0 & 0 & 0 \\
            	0 & 1 & 1 & 0 & 1 & 0 & 1 \\
            	0 & 1 & 0 & 1 & 1 & 1 & 0 \\
            	1 & 0 & 1 & 1 & 1 & 1 & 1 \\\hdashline
            	1 & 0 & 1 & 1 & 1 & 0 & 0 \\
            	0 & 0 & 0 & 0 & 1 & 0 & 1 \\
            	0 & 0 & 1 & 0 & 0 & 1 & 0 \\
            	1 & 1 & 1 & 0 & 0 & 1 & 1 \\\hdashline
            	0 & 1 & 1 & 1 & 0 & 0 & 0 \\
            	1 & 1 & 0 & 1 & 0 & 0 & 1 \\
            	1 & 1 & 0 & 0 & 1 & 1 & 0 \\
            	0 & 0 & 0 & 1 & 0 & 1 & 1 \\
            \end{array} \right)$$ %
			\caption{An orthogonal array of type OA$(12,7,2,2)$.}
			\label{tabbleOA12}
\end{table}
\end{proof}


We next generalize the above ideas to the case where linear independence is not assumed. 

\begin{lemma}
Let $m=2^{\alpha}$ where $\alpha\geq 2$ and 
$({\mathbb F}_2)^{\alpha} =\{{\bf e}_i\mid i\in [m]\}$. That is, label the binary vectors of dimension $\alpha$
with ${\bf e}_0$, ${\bf e}_1, \dots {\bf e}_{m-1}$. 
Let ${\bf v}_i$, $i\in [m]$ be any vectors (possibly non-distinct) of 
dimension $k\geq \alpha+m+1$.
Consider the $m\times 2^k$ array of vectors of dimension $k+\alpha$ given by $C\boxplus R$, 
where the rows of $R$ 
are the set of $2^k$ vectors with $0$ in the final $\alpha$ positions and the $i$th row of $C$ is given by 
${\bf v}_i\oplus {\bf e}_i$, where $ i\in [m]$.  
Then the elements in each column of $C\boxplus R$ can be rearranged so that the vectors in each row form an OA$(2^k,k+\alpha,2,2)$. 
\end{lemma}		

\begin{proof}
Let ${\bf e}_0={\bf 0}$ be the first row of $R$. 
Let $G$ be the subgroup generated by the set of vectors $\{{\bf v}_i: i\in [m]\}\cup \{{\bf 1}\}$. Then $|G|=2^{\ell}$ for some $1\leq \ell\leq m+1$.  
Let $G=\{{\bf g}_i\mid i\in [2^\ell]\}$ where
${\bf g}_0={\bf 0}$. 
Define a $2^\ell\times (k+\alpha)$ array $R_0$ so that row $i$ of $R_0$ is 
${\bf g}_i\oplus {\bf 0}$, $i\in [2^{\ell}]$.  Let $K_0$ be the 
$m\times 2^{\ell}$ array of vectors of dimension $k+\alpha$ given by
$C\boxplus R_0$. 

Next, 
let ${\bf z}_0+G$, ${\bf z}_1+G,\dots ,{\bf z}_{\beta-1}+G$ be the cosets of $G$ in 
$({\mathbb F}_2)^k$, where $\beta=2^{k-\ell}$.  
For each $j\in [\beta]$, let $K_{j}$ be formed from 
$K_0$ by adding ${\bf z}_{j}\oplus {\bf 0}$ to each vector in $K$. 
Next, cyclically permute the elements in each column of $K_{j}$ by $j$ places (modulo $m$) to create $K_j'$.

Note that $L=[K_0'|K_1'|\dots |K_{\beta-1}']$ is equal to $C\boxplus R$ after a permutation of the elements in each column. 
Moreover, let $S$ be the set of vectors of dimension $k+\alpha$ that occur in a given row of $L$.
Then we claim the following: 
\begin{enumerate}
    \item[(a)] For each ${\bf w}\in [2]^k$, the vector ${\bf w}\oplus {\bf e}_i\in S$, for some ${\bf e}_i\in ({\mathbb F}_2)^{\alpha}$;
    \item[(b)]  For each ${\bf e}_i\in ({\mathbb F}_2)^{\alpha}$, 
there are precisely $2^{k-\alpha}$ vectors in $S$ of the form 
${\bf w}\oplus {\bf e}_i$ for some ${\bf w}$; 
    \item[(c)] For each ${\bf e}_i\in ({\mathbb F}_2)^{\alpha}$, the vector ${\bf w}\oplus {\bf e}_i\in S$ if and only if 
    $({\bf w}+{\bf 1})\oplus {\bf e}_i\in S$.
\end{enumerate}
Similarly to Lemma \ref{strength2}, if this claim is true, it follows that 
the set of vectors in $S$ form the rows of an orthogonal array of strength $2$. So it suffices to show that the above claim is true. 

To see (a), let $j\in [\beta]$. Observe that in every row of $K_j$ and for every element ${\bf w}\in {\bf z}_j+G$, the vector ${\bf w}\oplus {\bf e}_i$ occurs in that row. The same property holds for $K_j'$. Next, from the conditions of the lemma, $\beta\geq m$; indeed $m$ divides $\beta$. Thus (b) is true. Finally (c), is true because ${\bf 1}\in G$.
\end{proof}

\begin{corollary} 
Let $\alpha\geq 2$ and 
$2^{\alpha}+\alpha+ 1\leq k$.
 Suppose there exists an  
OA$(2^{\alpha}b,k+{\alpha},2,2)$
such that the last $\alpha$ columns 
are an OA$(2^{\alpha}b,{\alpha},2,\alpha)$. 
 Then $I_{k+\alpha}(2^k,2^{\alpha}b,2,2)$ exists.
 \label{bigcases}
\end{corollary}	
	
\begin{proof}
We shall construct the transpose design 
$I_{k+\alpha}(2^{\alpha}b,k,2,2)$. 
Let $C$ be an OA$(2^{\alpha}b,k+\alpha,2,2)$
such that the last $\alpha$ columns 
are an OA$(2^{\alpha}b,{\alpha},2,\alpha)$. 
Note that $C$ can be partitioned into 
subarrays $C_0$, $C_1,\dots ,C_{b-1}$, each of dimension $2^{\alpha}\times (k+{\alpha})$, such that 
for each $i\in [b]$, 
the last 
${\alpha}$ columns of $C_i$ contain each element of $({\mathbb F}_2)^{\alpha}$ exactly once.  

Next, let $A$ be the $k \times (k+\alpha) $ matrix of the form 
$ [I|{\bf 0}]$ and let $R$ be an orthogonal array whose rows are the elements of $<A>$.
Observe that the rows of $R$ are the set of $2^k$ vectors with $0$ in the final $\alpha$ positions. 
Moreover, the nullspace of $A$ is generated by the columns of $A^\perp=[{\bf 0}|I]^T$.
Thus the columns of $CA^\perp$ are the last ${\alpha}$ columns of $C$. 

Thus, by Theorem $\ref{bigdeal1}$, $C\boxplus R$ is a row-column factorial design. Moreover, the columns of $C\boxplus R$ are each of strength $2$ since $C$ is of strength $2$. 
Finally, apply the previous lemma with $m=2^{\alpha}$ to rearrange the elements in each column of subarray $C_i\boxplus R$ so that 
$C\boxplus R$ becomes an  $I_{k+\alpha}(2^{\alpha}b,2^k,2,2)$. 
\end{proof}	

The following lemma uses a standard doubling technique. 

\begin{lemma}
Suppose there exists a Hadamard matrix of order $4b$ for some integer $b$. Let $\alpha\geq 2$. Then there exists an OA$(2^{\alpha}b,k+\alpha,2,2)$ such that the final $\alpha$ columns form an OA$(2^{\alpha}b,\alpha,2,\alpha)$, for any $\alpha\leq k+\alpha < 2^{\alpha}b$.   
\end{lemma}

\begin{proof}
We proceed by induction on $\alpha$. Suppose $\alpha=2$. 
 The existence of a Hadamard matrix of order $4b$ implies the existence of an OA$(4b,k+2,2,2)$ for any $k+2\leq 4b-1$ 
 by Lemmas \ref{haddd} and \ref{subarrays}. By the definition of the strength of an orthogonal array, the final two columns must contain each ordered pair $b$ times, so the final two columns form an OA$(4b,2,2,2)$. 
 
 Next assume that the lemma is true for a fixed value of $\alpha\geq 2$. 
 then there exists an 
 $L=$OA$(2^{\alpha}b,2^{\alpha}b-1,2,\alpha)$
 with the specified properties. 
 Observe that the following matrix $L'$ is an 
 OA$(2^{\alpha+1}b,2^{\alpha+1}b-1,2,2)$:
			$$ L'=\left[ \begin{array}{c|c|c}
            	 L & L & {\bf 1}  \\
            	\hline
            	 \overline{L} & L & {\bf 0} \\
            \end{array} \right],$$ %
 where $\overline{L}$ is formed from $L$ by replacing each $0$ with $1$. 
 Moreover, the final $\alpha+1$ columns  of $L'$ contain each binary sequence of dimension $\alpha+1$ exactly once. Thus the final $\alpha+1$ columns form an OA$(2^{\alpha+1}b,\alpha+1,2,\alpha+1)$. Hence an OA$(2^{\alpha+1}b,k+\alpha+1,2,\alpha+1)$ can be obtained for any $k$ such that $k+\alpha+1< 2^{\alpha+1}b$ by deletion of columns. This completes the induction and the proof. 
\end{proof}


From the previous lemma and Corollary \ref{bigcases}, we have the following. 

\begin{theorem}
If there exists a Hadamard matrix $H(4b)$, then there exists 
$I_{k+\alpha}(2^{\alpha}b,2^k,2,2)$ for any $2\leq \alpha$;   
$2^{\alpha}+\alpha+ 1\leq k< 2^{\alpha}b-\alpha$.
\label{biggercases}
\end{theorem}


Before we completely deal with the case when $m$ is odd and $k$ is small, we need some constructions for specific parameters. 

\begin{lemma}
There exists $I_5(12,8,2,2)$, $I_6(12,16,2,2)$ and      
 $I_{4}(12,12,2,2)$.
\label{specifics}
\end{lemma}

\begin{proof}
In Table \ref{tabble3} we present an abelian array of the form $C\boxplus R$, where 
$R$ is the rowspace of the $3\times 5$ matrix 
$[I_3|0]$ and $C$ is an OA$(12,5,2,2)$ (constructed from $5$ columns of a Hadamard matrix of order $12$). 
Now, the columns of the matrix $C[I_3|0]^T$ are in turn distinct columns of $C$; thus $C[I_3|0]^T$ is an 
OA$(12,3,2,2)$. 
Hence, by Theorem  \ref{bigdeal1}, $C\boxplus R$ is a row-column factorial design. Moreover, each column is an orthogonal array of strength $2$. We can then rearrange the elements in each column to create an $I_5(12,8,2,2)$; the rearrangement is indicated by the use of superscripts. That is, we permute entries within each column so that vectors with the same superscript belong to the same row.

\begin{table}[H]
	\centering
			\renewcommand{\arraystretch}{1.3}
			\resizebox{12cm}{!}{%
			\begin{tabular}{c}
			$\begin{array}{|cccc cccc|}
			    \hline
            	\bf 00000^{\textit{A}}&	\bf{10000}^\textit{C}&	{\bf 01000}^\textit{D} &	\bf11000^\textit{B} &	\bf{00100}^\textit{D} &	\bf{10100}^\textit{B} &	\bf 01100^\textit{A}&	\bf{11100}^\textit{C} \\
                \bf{00001}^{\textit{B}} &	 10001^{D} &	01001^C&	11001^A&	00101^C&	10101^A&	 01101^{B} &	 11101^{D}  \\
                \bf{00110}^\textit{C}&	10110^A&	 01110^{B} &	 11110^{D} &	 00010^{B} &	 10010^{D} &	01010^C&	11010^A \\
                \bf{01011}^\textit{D} &	 11011^{B} &	00011^A&	10011^C&	01111^A&	11111^C&	00111^{D}&	 10111^{B}  \\ \hline
                \bf{01100}^\textit{E}&	11100^G&	00100^{F}&	10100^{H}&	01000^{H}&	11000^{F}&	00000^G&	{10000}^E \\
                \bf{01111}^{\textit{F}}&	11111^{H}&	00111^E&	10111^G&	01011^G&	11011^E&	00011^{H}&	10011^{F} \\
                \bf{11001}^\textit{G}&	01001^E&	10001^{H}&	00001^{F}&	11101^{F}&	01101^{H}&	10101^E&	00101^G \\
                \bf{11010}^{\textit{H}}&	01010^{F}&	10010^G&	00010^E&	11110^E&	01110^G&	10110^{F}&	00110^{H} \\ \hline
                \bf{10010^\textit{I}}&	00010^K&	11010^{J}&	01010^{L}&	10110^{J}&	00110^{L}&	11110^I&	01110^K \\
                \bf{10101}^{\textit{J}}&	00101^{L}&	11101^I&	01101^K&	10001^I&	00001^K&	11001^{J}&	01001^{L} \\
                \bf{10111^\textit{K}}&	00111^I&	11111^{L}&	01111^{J}&	10011^{L}&	00011^{J}&	11011^K&	01011^I \\
                \bf{11100}^{\textit{L}}&	01100^{J}&	10100^K&	00100^I&	11000^K&	01000^I&	10000^{L}&	00000^{J} \\ \hline
            \end{array}$ 
            \end{tabular}%
			}
			\caption{An array of type $I_5(12,8,2,2)$, with rows indicated by superscripts.}
			\label{tabble3}
\end{table}

In Table \ref{tabble4}, 
first consider the matrix $A$ formed by 
the first $4$ columns. 
This matrix is abelian of the form $C\boxplus R$, where $C$ is an OA$(12,4,2,2)$ and
$R$ is the rowspace of a $2\times 4$ matrix. 
 By Theorem  \ref{bigdeal1} (or inspection if easier), $C\boxplus R$ is a row-column factorial design. 
Thus $[A|A|A]$, as shown in Table \ref{tabble4}, is a row-column factorial design with each column an orthogonal array of strength $2$.
It thus remains to rearrange the elements within each column so that the rows are each of strength $2$. 
The superscripts $A$, $B$, $C$ and $D$ 
 indicate $4$ rows of strength $2$. The remaining rows are formed by cyclic shifts of each of these by $4$ rows and then $8$ rows; also indicated by superscripts.  
This results in the array given in Table \ref{tabble5}. 

\begin{table}[H]
	\begin{center}
		\renewcommand{\arraystretch}{1.4}
		\resizebox{\columnwidth}{!}{%
		\begin{tabular}{c}
		$ \begin{array}{|cccc| cccc |cccc|}
		    \hline
            \bf 0000^\textit{A}& \bf 0001^{\textit{B}}& \bf 1111^{\textit{C}}& \bf 1110^{\textit{D}}& \bf 0000^\textit{E}& \bf 0001^\textit{F}& \bf 1111^\textit{G}& \bf 1110^\textit{H}& \bf 0000^\textit{I}& \bf 0001^\textit{J}&	            \bf 1111^\textit{K}&	            \bf 1110^\textit{L} \\
                \bf 0110^\textit{J}&	0111^I&	1001^A&	1000^{B}&	0110^{B}&	0111^A&	1001^E&	1000^F&	0110^F&	0111^E&	1001^I&	1000^J \\
                \bf 1011^\textit{B}&	1010^A&	0100^{D}&	0101^{C}&	1011^F&	1010^E&	0100^H&	0101^G&	1011^J&	1010^I&	0100^L&	0101^K \\
                \bf 1101^\textit{D}&	1100^{C}&	0010^{B}&	0011^A&	1101^H&	1100^G&	0010^F&	0011^E&	1101^L&	1100^K&	0010^J&	0011^I \\ \hline
                \bf 0000^\textit{L}&	0001^K&	1111^K&	1110^L&	0000^{D}&	0001^{C}&	1111^{C}&	1110^{D}&	0000^H&	0001^G&	1111^G&	1110^H \\
                \bf 0111^\textit{H}&	0110^G&	1000^G&	1001^H&	0111^L&	0110^K&	1000^K&	1001^L&	0111^{D}&	0110^{C}&	1000^{C}&	1001^{D} \\
                \bf 0011^\textit{C}&	0010^{D}&	1100^I&	1101^J&	0011^G&	0010^H&	1100^A&	1101^{B}&	0011^K&	0010^L&	1100^E&	1101^F \\
                \bf 0101^\textit{I}&	0100^J&	1010^J&	1011^I&	0101^A&	0100^{B}&	1010^{B}&	1011^A&	0101^E&	0100^F&	1010^F&	1011^E \\ \hline
                \bf 1001^\textit{K}&	1000^L&	0110^E&	0111^F&	1001^{C}&	1000^{D}&	0110^I&	0111^J&	1001^G&	1000^H&	0110^A&	0111^{B} \\
                \bf 1110^\textit{E}&	1111^F&	0001^F&	0000^E&	1110^I&	1111^J&	0001^J&	0000^I&	1110^A&	1111^{B}&	0001^{B}&	0000^A \\
                \bf 1010^\textit{G}&	1011^H&	0101^L&	0100^K&	1010^K&	1011^L&	0101^{D}&	0100^{C}&	1010^{C}&	1011^{D}&	0101^H&	0100^G \\
                \bf 1100^\textit{F}&	1101^E &	0011^H &	0010^G &	1100^J &	1101^I &	0011^L &	0010^K &	1100^{B} &	1101^A&	0011^{D}&	0010^{C} \\ \hline
            \end{array} $ 
            \end{tabular}%
			}
			\caption{A factorial row-column design with each row strength $2$.}
			\label{tabble4}
			\end{center}
\end{table}

\begin{table}[H]
	\begin{center}
	\renewcommand{\arraystretch}{1.2}
		\resizebox{\columnwidth}{!}{%
		\begin{tabular}{c}
			$ \begin{array}{|cccc| cccc |cccc|}
		    \hline
		    \bf 0000&   \bf 1010&	\bf 1001&	\bf 0011&	\bf 0101&	\bf 0111&	\bf 1100&	\bf 1011&	\bf 1110&	\bf 1101&	\bf 0110&	\bf 0000 \\
            \bf 0110&	0100&	1010&	1101&	1100&	1111&	0001&	0111&	1011&	0001&	0010&	1000 \\
            \bf 1011&	0001&	0010&	1000&	0110&	0100&	1010&	1101&	1100&	1111&	0001&	0111 \\
            \bf 1101&	0010&	0100&	1110&	0000&	1000&	0101&	1110&	0111&	1011&	0011&	1001 \\ \hline
            \bf 0000&	1000&	0101&	1110&	0111&	1011&	0011&	1001&	1101&	0010&	0100&	1110 \\
            \bf 0111&	1011&	0011&	1001&	1101&	0010&	0100&	1110&	0000&	1000&	0101&	1110 \\
            \bf 0011&	1100&	1111&	0101&	1001&	0001&	1111&	0100&	1010&	0110&	1000&	0010 \\
            \bf 0101&	0111&	1100&	1011&	1110&	1101&	0110&	0000&	0000&	1010&	1001&	0011 \\ \hline
            \bf 1001&	0001&	1111&	0100&	1010&	0110&	1000&	0010&	0011&	1100&	1111&	0101 \\
            \bf 1110&	1101&	0110&	0000&	0000&	1010&	1001&	0011&	0101&	0111&	1100&	1011 \\
            \bf 1010&	0110&	1000&	0010&	0011&	1100&	1111&	0101&	1001&	0001&	1111&	0100 \\
            \bf 1100&	1111&	0001&	0111&	1011&	0001&	0010&	1000&	0110&	0100&	1010&	1101 \\ \hline

            \end{array} $ 
            \end{tabular}%
			}
	        \caption{An array of type $I_4(12,12,2,2)$.}
	        \label{tabble5}
	\end{center}
\end{table}

Finally, $I_6(12,16,2,2)$ is given in the Appendix. Similarly to above, this is presented first as an abelian row-column factorial design where each column is of strength $2$. The superscripts indicate how to permute the entries within each column. 
\end{proof}

We can now give necessary and sufficient conditions for the case when the number of rows is congruent to $4$ (mod $8$), assuming the truth Conjecture \ref{whoknows}.  
	
\begin{theorem}
Let $m$ and $b$ be odd.  
If Conjecture \ref{whoknows} is true, 
Then $I_k(4m,2^ab,2,2)$ exists if and only if 
$(k,4m,2^ab,2,2)$ is admissible and 
$$(k,4m,2^ab,2,2)\not\in \{(3,4m,4,2,2),(3,4,4m,2,2)\mid m \mbox{\rm\ is odd}\}.$$

\label{thisisalabel}
\end{theorem}

\begin{proof}
Since $(k,4m,2^ab,2,2)$ is admissible, from Lemmas \ref{triviality} and \ref{2strong}:  
$a\geq 2$, $k\leq a+2$,  
$k\leq 4m-1$ and $k\leq 2^ab-1$. 

{\bf Case 1}: $a=2$ and $b=1$. Then $k\leq 3$. 
Suppose $k=2$. 
Now, $[00,01,10,11]^T$ is an OA$(4,2,2,2)$, so by Lemma \ref{letsgettrivial}, there exists $I_2(4,4,2,2)$.  
Thus by Corollary \ref{cor:blowup}, $I_2(4m,4b,2,2)$ exists for any integers $m$ and $b$.
Otherwise $k=3$. 
By Lemma \ref{thm:exception1}, $I_3(4m,4,2,2)$ does not exists for odd $m$. 

{\bf Case 2}: $a=2$ and $b\geq 3$.
If $m=1$, this is the transpose of Case 1, so we may assume $m\geq 3$. 
Thus $k\leq 4$ implies admissibility. 
From Lemma \ref{specifics}, there exist 
$I_4(12,12,2,2)$ and $I_5(12,8,2,2)$. From Theorem \ref{biggerthanelvis}, $I_6(8,8,2,2)$ exists.
In turn, by Lemma \ref{subarrays}, $I_4(12,8,2,2)$ and $I_4(8,8,2,2)$ exist.
By adjoining copies of 
$I_4(12,12,2,2)$, 
$I_4(12,8,2,2)$, $I_4(8,12,2,2)$ and $I_4(8,8,2,2)$
as needed using Lemma \ref{lem:glueing}, there exists 
$I_4(4m,4b,2,2)$ for any $m,b\geq 3$. 
{\bf Case 3}: $m=1$ and $a\geq 3$.
Since $m=1$, $k\leq 3$. Then there exists $I_3(4,8,2,2)$ by Theorem \ref{biggerthanelvis}. Thus there exists
$I_3(4,2^ab,2,2)$ for any $a\geq 3$ by Corollary \ref{cor:blowup}. 

{\bf Case 4}: $m\geq 3$ and $a\in \{3,4\}$.  
Here $k\leq a+2$ implies admissibility. 
Now, $I_{a+2}(12,2^a,2,2)$ exists for each $a\in \{3,4\}$ from Lemma  \ref{specifics}. 
Next, $I_6(8,8,2,2)$ exists by Theorem \ref{biggerthanelvis}. 
Thus by Lemma \ref{lem:glueing},  
 $I_{a+2}(4m,2^a,2,2)$ exists for any odd integer $m$. 
In turn, $I_{a+2}(4m,2^ab,2,2)$ exists by Corollary \ref{cor:blowup}. 

{\bf Case 5}: $m\geq 3$ and $a=5$.  
Then $k\leq 7$ implies admissibility. 
Now, $I_7(12,2^5,2,2)$ 
exists by Corollary \ref{yesyes2}. 
Also, $I_7(8,2^5,2,2)$ exists by Theorem \ref{biggerthanelvis}. 
Thus by Lemma \ref{lem:glueing} and Corollary \ref{cor:blowup}, 
$I_7(4m,2^5b,2,2)$ exists for all odd $m\geq 3$ and odd $b$.

{\bf Case 6}: $m\geq 3$ and $a\geq 6$.  
From Corollary \ref{eightandabove}  and assuming the truth of 
Conjecture \ref{whoknows}, 
$I_{a+2}(4m,2^{a},2,2)$ exists for all $6\leq a\leq 4m-3$.     
Thus $I_k(4m,2^{a}b,2,2)$ exists for all $k\leq a+2$.  
\end{proof}

\begin{theorem}
Let $m\leq 5$ and $b$ odd. Then $I_k(4m,2^ab,2,2)$ exists for all admissible 
    $$(k,4m,2^ab,2,2)\not\in \{(3,4m,4,2,2),(3,4,4m,2,2)\mid m \mbox{\rm\ is odd}\}.$$
    \label{section5main}
\end{theorem}

\begin{proof}
From the previous theorem and the fact that Conjecture 
\ref{whoknows} is true for $m\in \{3,5\}$, we can assume  $m\in \{2,4\}$.   

Let $m=2$. 
 By Theorem \ref{biggerthanelvis} there exists $I_3(8,4,2,2)$ and  $I_7(8,8,2,2)$. 
Also there exists $I_5(4b,8,2,2)$ (and thus $I_5(8,4b,2,2)$) by the previous theorem, where $b\geq 3$ is odd. 
The result then follows by Lemma \ref{lem:glueing} and Corollary \ref{cor:blowup}. 

Otherwise $m=4$. 
Then by Theorem \ref{biggerthanelvis} there exists $I_3(16,4,2,2)$ and $I_{a+4}(16,2^{a},2,2)$ for any $3\leq a\leq 11$. 
Also there exists $I_6(16,4b,2,2)$, where $b\geq 3$ is odd, by the previous theorem. 
The result then follows by Lemma \ref{lem:glueing} and Corollary \ref{cor:blowup}. 
\end{proof}

\section{Binary row-column factorial designs with strength $t=3$}

In this section we restrict ourselves to binary row-column factorial designs 
of strength $3$. We completely classify these when the dimensions of the arrays are powers of $2$. The aim of this section is to prove the following theorem. 

\begin{theorem}
Let $M\leq N$. Then an array of type 
$I_k(2^M,2^N,2,3)$ exists if and only if 
$3\leq k\leq M+N$, $3\leq M$, $k\leq 2^{M-1}$ and 
 $(k,M,N)\not\in \{(4,3,3),(8,4,4)\}$.
\label{strength3}     
\end{theorem}

\begin{lemma}
Let $M\leq N$. Then $(k,2^M,2^N,2,3)$ is admissible if and only if 
$3\leq k\leq M+N$, $3\leq M$ and $k\leq 2^{M-1}$. 
\label{strengththree}
\end{lemma}

\begin{proof}
By Lemma \ref{triviality}, 
 $3\leq k\leq M+N$ and $3\leq M$.   
The bound $k\leq 2^{M-1}$ (and sufficiency) follows by Lemma \ref{3strong}. 
\end{proof}

To establish the two exceptions in  
Theorem \ref{strength3}, we first need the following result on orthogonal arrays. This result is a fairly standard observation for researchers in Hadamard codes but we include a proof for thoroughness. 

\begin{lemma}
Let $M\geq 3$. 
In any OA$(2^M,2^{M-1},2,3)$, the weight of any two rows has the same parity. 
\label{strength3conditions}
\end{lemma}

\begin{proof}
Let $K$ be an OA$(2^M,2^{M-1},2,3)$. 
Without loss of generality  assume that, restricting ourselves to the first two  columns of $K$,  
the first $2^{M-2}$ rows contain the ordered pairs $(1,1)$, the next $2^{M-2}$ rows contain the ordered pairs $(1,0)$, the next $2^{M-2}$ rows contain the ordered pairs $(0,1)$ and the final $2^{M-2}$ rows contained the ordered pairs $(0,0)$. 

 For the rest of the proof, we assume that in a Hadamard matrix each $-1$ has been replaced by $0$. 
It follows, from the strength $3$ property of the orthogonal array, that: (a) the first $2^{M-1}$ rows form a Hadamard matrix; 
(b) the last $2^{M-1}$ rows form a Hadamard matrix; and
(c) the first $2^{M-2}$ rows
together with the third set of  $2^{M-2}$ rows forms a Hadamard matrix. 

Now, in a normalized Hadamard matrix of order at least $4$, the weight of any row or column is even. Equivalent Hadamard matrices are formed by rearranging rows or columns, taking a transpose or swapping $0$ with $1$ in any row or column. All of these equivalences preserve the property that the weight of each pair of rows shares the same parity. The result follows.  
\end{proof}

\begin{corollary}
There exists neither  
an array of type $I_4(8,8,2,3)$ nor an array of type  
$I_8(16,16,2,3)$. 
\label{strength3exceptions}
\end{corollary}

\begin{proof}
If an array of type $I_4(8,8,2,3)$ exists, then the vectors in any row or column, by definition, form an OA$(8,4,2,3)$. Thus, from the previous lemma, the weight of every vector in the array has the same parity. Hence the vectors in all the cells of the array do not form a factorial design. 
Similarly, there does not exist an array of type $I_8(16,16,2,3)$. 
\end{proof}

We now focus on proving Theorem \ref{strength3} in the case where $M\geq 5$.     
We will use Theorem \ref{thm:FR.Polynomilas.existence.} for this case. We first need some preliminary lemmas.

We remind the reader that a set $S$ of vectors is $t$-independent if and only if each subset of $S$ of size $t$ is independent. 

\begin{lemma} \label{lem:circulant set weight 4}
	\sloppy Let $ C $ be a set consisting of $ M $ cyclic permutations of the vector $ (1,1,1,1,0,0, \dots , 0) $ over $ \mathbb{F}_2^M $, where $ M \geq 5 $ and $ M \neq 6 $. Then the vectors in $ C $ are $ 3-$independent.
\end{lemma}

\begin{proof}
	Note that all the vectors in $ C $ have weight 4. Also, any three vectors ${\bf t}, {\bf u}, {\bf v}$ in $ (\mathbb{F}_2)^M $ are linearly dependent if and only if  ${\bf u} + {\bf v} = {\bf t}$. Now for any two vectors ${\bf u}$ and ${\bf v}$ in $ C $ we have the following possibilities:
	
	
	\textbf{Case I}: There is at most one $ i $ such that $ u_i = v_i = 1 $.
	In this case $ \omega(u+v) = 6 $ and therefore $ {\bf u}+{\bf v} \not \in C$.
	
	\textbf{Case II}: There are exactly two values of $ i $ for which $ u_i = v_i = 1 $.
	In this case $ \omega({\bf u} + {\bf v}) = 4 $ and $M\geq 7$ (since $M\neq 6$).  
	However, notice that the vector ${\bf u} + {\bf v}$ contains the values 
	$1,1,0,0,1,1,0$ at seven consecutive positions (modulo $n$) and therefore does not belong to $C$.
	
	\textbf{Case III}: There are exactly three values of $ i $ for which $ u_i = v_i = 1$. In this case $ \omega({\bf u} + {\bf v}) = 2$.
\end{proof}

\begin{corollary}
	For $ M \geq 5  $ and $ M \neq 6 $, let $ B =  \{e_1, \dots , e_M\} $ be the standard basis for $ \mathbb{F}_2^M $ and $ C $ be the set defined in Lemma \ref{lem:circulant set weight 4}. Then the set $ W = B \cup C $ is $ 3-$independent.
\end{corollary}

\begin{lemma}
	\sloppy For $N\geq M \geq 5 $ there exists an $ I_k(2^M,2^N,2,3)$
	if and only if $(k,2^M,2^N,2,3)$ is admissible. 
\label{5orbigger}
\end{lemma}

\begin{proof}
	From Lemma \ref{strengththree} and Corollary \ref{cor:blowup}, it suffices to assume $k = \min \{ 2^{M-1}, M+N \}$.
	
	We split the proof into different cases. In each case we define a matrix $L$  
	satisfying the required conditions of Theorem \ref{thm:FR.Polynomilas.existence.}. 

	\textbf{Case I}: $ M \neq 6 $. Let $ C_M $ be an $M\times M$ matrix such that  the rows are the elements of the set $C$ defined in Lemma \ref{lem:circulant set weight 4} with the main diagonal of $C_M$ containing only entry $1$. 
	Let $ D $ be the set of all the rows in $ I_M $ and $ C_M - I_M $. Let $ W $ be the set of all vectors of odd weight in $ \mathbb{F}_2^M $. By Lemma \ref{3strong}, $W$ is a 3-independent set of vectors. Let 
	$S$ be a $(k-2M)\times M$ matrix such that each row is a distinct element in $ W \setminus D $.
Then the matrix $L$ is as follows. 
	\begin{equation} \label{eq:matrix.L}
		L = \left(\renewcommand\arraystretch{1.8} \begin{tabular}{>{$} P{2.0cm} <{$} | >{$} P{2.0cm} <{$} : >{$} P{2.0cm} <{$} }
			I_M		 & I_M & \textbf{0} \\ \hdashline
			C_M-I_M  & C_M & \textbf{0} \\ \hdashline 
			S & \textbf{0} & I_{k-2M} \\
		\end{tabular}\right)
	\end{equation}

	\textbf{Case II}: When $ M = 6 $. In this case we can take the above matrix $ L $ using the following $ C_M $:
	 \begin{equation*}
		C_M = \left(\begin{array}{cccccc}
			1 & 1 & 1 & 1 & 0 & 0 \\
			0 & 1 & 1 & 1 & 1 & 0 \\
			0 & 0 & 1 & 1 & 1 & 1 \\
			1 & 0 & 0 & 1 & 1 & 1 \\
			1 & 0 & 1 & 0 & 1 & 1 \\
			1 & 0 & 1 & 1 & 0 & 1 \\
		\end{array} \right).
	\end{equation*}
\end{proof}

\begin{lemma}
    There exists  $I_3(8,8,2,3)$, $I_4(8,16,2,3)$, $I_7(16,16,2,3)$ and $I_8(16,32,2,3)$.
\end{lemma}

\begin{proof}
    The eight binary vectors of dimension $3$ give the rows of an OA$(8,3,2,3)$. Thus by 
Lemma \ref{letsgettrivial}, there exists an array of type $I_3(8,8,2,3)$.
Next, the $8$ binary vectors of dimension $4$ and even weight give the rows of an OA$(8,4,2,3)$.
Using Lemma \ref{letsgettrivial} again, there exists an array of type   $I_4(8,16,2,3)$.

By Theorem \ref{thm:FR.Polynomilas.existence.} and the following array $L$, there exists  
$I_8(16,32,2,3)$. Moreover, we get $I_7(16,16,2,3)$ for free by deleting the last row and column, as indicated by dotted lines.  
\begin{equation*}
		L = \left(\begin{array}{cccc|cccc: c}
			1&0&0&0&1&0&0&0&0 \\
			0&1&0&0&0&1&0&0&0 \\
			0&0&1&0&0&0&1&0&0 \\
			0&0&0&1&0&0&0&1&0 \\ 
			1&1&1&0&1&1&0&1&0 \\
			1&1&0&1&1&0&1&1&0 \\
			1&0&1&1&0&1&1&1&0 \\ \hdashline
			0&1&1&1&1&1&1&0&1 \\
		\end{array} \right).
	\end{equation*}
\end{proof}

We now have all the tools, base cases and exceptions we need to prove Theorem \ref{strength3}. 
From Lemma \ref{5orbigger},  
  we may restrict ourselves to the case $M\in \{3,4\}$ and $N\geq M$. 
If $M=3$ then by Lemma \ref{strengththree}, 
$3\leq k\leq 2^{M-1}=4$. 
An $I_3(8,8,2,3)$ exists by the previous lemma. 
Thus by Corollary \ref{cor:blowup}, there exists   
 $I_3(2^M,2^N,2,3)$ whenever $M,N\geq 3$.
Next, $I_4(8,8,2,3)$ does not exist by Corollary \ref{strength3exceptions}. 
However $I_4(8,16,2,3)$ exists by the previous lemma. 
Thus by Corollary \ref{cor:blowup} there exists an array of type $I_4(2^M,2^N,2,3)$ whenever $M\geq 3$ and $N\geq 4$.
 
Finally suppose that $N\geq M=4$. 
 By Lemma \ref{strengththree},  $3\leq k\leq 8$. Now, $I_8(16,16,2,3)$ does not exist by Corollary \ref{strength3exceptions} but  $I_7(16,16,2,3)$ exists by the previous lemma and $I_8(16,32,2,3)$. 
The result then follows by Corollary \ref{cor:blowup}. 

\section{Conclusion}

We first discuss some limitations to the approach given in Section 5.  
Firstly, the idea in Theorem \ref{eightisenough} cannot work for theoretical reasons when $k\leq 6$, and for  computational reasons (inspection of possible cases)  when $k=7$ and $m\in \{3,5\}$. 
 The reason that $k\geq 7$ is necessary for the approach is as follows. 
 The counting argument in the following lemma shows that if $m$ is odd, then
in any OA$(4m,n,2,2)$, if the sum of $\ell$ columns has weight $2m$ (that is, contains $2m$ occurrences of $1$), then $\ell\equiv 1$ or $2$ (mod $4$). In turn, from Corollary \ref{conditionsonk}, the columns of $K$ must be distinct and have weight at least $2$. 
This precludes a suitable $K$ for $k\leq 6$. 

\begin{lemma}
	Let $ H $ be an OA$(4m,\ell,2,2)$, where $ m $ is an odd integer. Let 
	$ \textbf{v} $ be the sum of columns of $ H $ over ${\mathbb F}_2$. 
	\begin{itemize}
		\item If $\ell \equiv 0$  or  $3 \pmod{4} $ then $\omega( \textbf{v}) \equiv 0 \pmod{4} $.
		\item If $\ell \equiv 1$ or  $2 \pmod{4} $ then $\omega( \textbf{v}) \equiv 2 \pmod{4} $.
	\end{itemize}  
 \label{countingargument}
\end{lemma}

\begin{proof}
	Let $ x_i $ be the weight of the $i$th row of $ H $. The total number of $(1,1)$ pairs such that both of them lie in the same row is given by $ \sum_{i=1}^{4m} \binom{x_i}{2} $. Also, since each pair of columns contain exactly $ m $ $(1,1)$ pairs, therefore this number can also be given by $ m \binom{\ell}{2} $. Thus we have,
	\begin{equation*}
		\sum_{i=1}^{4m} \binom{x_i}{2} = m \binom{\ell}{2}
	\end{equation*}
	which simplifies to:
	\begin{equation} \label{eq:x_i^2}
		\sum_{i=1}^{4m} x_i^2 = m\ell(\ell+1). 
	\end{equation}
	
	Also notice that, $ x_i^2 \equiv 1 \pmod{4}$ if $ x_i $ is odd and  $ x_i^2 \equiv 0 \pmod{4}$ otherwise. Thus we have:
 $$\omega({\mathbf v}) = 
  \sum_{i=1}^{4m} (x_i \pmod{2}) 
  = \sum_{i=1}^{4m} (x_i^2 \pmod{2}) 
  = \sum_{i=1}^{4m} (x_i^2 \pmod{4}).$$  
   The result is now follows from (\ref{eq:x_i^2}) and the fact that $ m $ is odd.
\end{proof}

We have intentionally structured our paper so that abelian and non-abelian constructions are distinguished. By inspection, we have determined that there does not exist an abelian $I_5(12,8,2,2)$; a non-abelian example is given in Lemma \ref{specifics}.
 However we do not know at this stage whether there are infinitely many parameters for which their exists only a non-abelian binary strength $2$ row-column factorial design. 

As observed in the introduction, finding necessary and sufficient conditions for the existence of a strength $2$ binary row-column factorial design depends on the Hadamard conjecture. 
However, the following may be more within reach.
\begin{conjecture}
If there exists a Hadamard matrix of order $4m$, then there exists an $I_k(4m,4n,2,2)$ for any $n\geq m$, $2^k|16mn$ and $k\leq 4m-1$, with the exception $m=1$ and $n$ is odd.  
\end{conjecture}

From Theorem \ref{section5main}, the above conjecture is true for $m\leq 5$.  
 For $m=6$, the unknown cases with minimal parameters are:   
 $I_{k+3}(24,2^k,2,2); 5\leq k\leq 11$. 
 From Theorem \ref{biggercases} and the existence of a Hadamard matrix of order $12$,  
  $I_{k+3}(24,2^k,2,2)$ exists for $12\leq k\leq 20$.
  



\begin{landscape}

\appendix

\section{$I_6(12,16,2,2)$}
\begin{table}[H] 
		\begin{center}
			\renewcommand{\arraystretch}{2}
			\resizebox{\columnwidth}{!}{%
			\begin{tabular}{c}
			$\begin{array}{|cccc cccc cccc cccc|}
			\hline
\bf {000000}^\textit{A}&	\bf 100000^{\textit{B}}&	\bf 010000^\textit{C}&	\bf 001000^{\textit{D}}&	\bf {000100}^\textit{A}&	\bf 001100^{\textit{D}}&	\bf 010100^\textit{C}&	\bf 011000^{\textit{B}}&	\bf 101000^\textit{C}&	\bf 100100^{\textit{B}}&	\bf 110000^{\textit{D}}&	\bf 011100^{\textit{B}}&	\bf 101100^\textit{C}&	\bf 110100^{\textit{D}}&	\bf {111000}^\textit{A}&	\bf {111100}^\textit{A} \\
\bf 110101^{\textit{B}}&	{010101}^A&	100101^{D}&	111101^C&	110001^{B}&	111001^C&	100001^{D}&	{101101}^A&	011101^{D}&	{010001}^A&	000101^C&	{101001}^A&	011001^{D}&	000001^C&	001101^{B}&	001001^{B} \\
\bf 110010^\textit{C}&	010010^{D}&	{100010}^A&	111010^{B}&	110110^C&	111110^{B}&	{100110}^A&	101010^{D}&	{011010}^A&	010110^{D}&	000010^{B}&	101110^{D}&	{011110}^A&	000110^{B}&	001010^C&	001110^C \\
\bf 000011^{\textit{D}}&	100011^C&	010011^{B}&	{001011}^A&	000111^{D}&	{001111}^A&	010111^{B}&	011011^C&	101011^{B}&	100111^C&	{110011}^A&	011111^C&	101111^{B}&	{110111}^A&	111011^{D}&	111111^{D} \\ \hline
\bf 001100^\textit{E}&	101100^E&	011100^H&	000100^F&	001000^G&	000000^H&	011000^F&	010100^G&	100100^F&	101000^G&	111100^H&	010000^E&	100000^H&	111000^F&	110100^G&	110000^E \\
\bf 010101^\textit{F}&	110101^H&	000101^E&	011101^E&	010001^H&	011001^G&	000001^G&	001101^F&	111101^G&	110001^F&	100101^G&	001001^H&	111001^E&	100001^E&	101101^H&	101001^F \\
\bf 011110^\textit{G}&	111110^F&	001110^G&	010110^H&	011010^E&	010010^F&	001010^E&	000110^H&	110110^E&	111010^H&	101110^F&	000010^F&	110010^G&	101010^H&	100110^E&	100010^G \\
\bf 011011^\textit{H}&	111011^G&	001011^F&	010011^G&	011111^F&	010111^E&	001111^H&	000011^E&	110011^H&	111111^E&	101011^E&	000111^G&	110111^F&	101111^G&	100011^F&	100111^H \\ \hline
\bf 111000^\textit{I}&	011000^J&	101000^K&	110000^J&	111100^K&	110100^L&	101100^I&	100000^L&	010000^I&	011100^L&	001000^L&	100100^J&	010100^K&	001100^J&	000000^K&	000100^I \\
\bf 101001^\textit{J}&	001001^I&	111001^I&	100001^L&	101101^L&	100101^J&	111101^K&	110001^K&	000001^K&	001101^K&	011001^J&	110101^I&	000101^I&	011101^L&	010001^L&	010101^J \\
\bf 100110^\textit{K}&	000110^L&	110110^L&	101110^I&	100010^I&	101010^K&	110010^J&	111110^J&	001110^J&	000010^J&	010110^K&	111010^L&	001010^L&	010010^I&	011110^I&	011010^K \\
\bf 101111^\textit{L}&	001111^K&	111111^J&	100111^K&	101011^J&	100011^I&	111011^L&	110111^I&	000111^L&	001011^I&	011111^I&	110011^K&	000011^J&	011011^K&	010111^J&	010011^L \\ \hline
        \end{array}$
        \end{tabular}%
			}
			\caption{An array of type $I_6(12,16,2,2)$.}
			
		\end{center}
	\end{table}	

\end{landscape}

\end{document}